\documentclass[reqno]{amsart}

\usepackage{amsmath,amssymb,amsthm,amsfonts,mathtools}
\usepackage{hyperref}
\usepackage[english]{babel}
\usepackage{cancel}
\usepackage{xcolor}
\usepackage{csquotes}
\usepackage{inputenc}
\usepackage{fontenc}
\usepackage{tikz}

\newcommand{\cA}{{\mathcal A}}
\newcommand{\cB}{{\mathcal B}}
\newcommand{\cD}{{\mathcal D}}

\newcommand{\cG}{{\mathcal G}}
\newcommand{\cH}{{\mathcal H}}
\newcommand{\cI}{{\mathcal I}}
\newcommand{\cK}{{\mathcal K}}
\newcommand{\cL}{{\mathcal L}}

\newcommand{\cQ}{{\mathcal Q}}

\newcommand{\cU}{{\mathcal U}}

\newcommand{\myspace}{\qquad\qquad\qquad}

\newtheorem{theorem}{Theorem}[section]
\newtheorem{lemma}[theorem]{Lemma}
\newtheorem{proposition}[theorem]{Proposition}
\newtheorem{remark}[theorem]{Remark}
\newtheorem{remarks}[theorem]{Remarks}

\newtheorem{assumptions}[theorem]{Assumptions}

\newtheorem{problem}[theorem]{Problem}
\numberwithin{equation}{section}
\date{}

\begin{document}

\title[]{Optimal synthesis control for evolution equations subject to nonlocal inputs}

\author{Paolo Acquistapace}
\address{Paolo Acquistapace, Universit\`a di Pisa ({\em Ret.}), Dipartimento di Matematica,
Largo Bruno~Pontecorvo 5, 56127 Pisa, ITALY 
}
\email{paolo.acquistapace(at)unipi.it}

\author{Francesca Bucci}
\address{Francesca Bucci, Universit\`a degli Studi di Firenze,
Dipartimento di Matematica e Informatica,
Via S.~Marta 3, 50139 Firenze, ITALY
}
\email{francesca.bucci(at)unifi.it}

\makeatletter
\@namedef{subjclassname@2020}{\textup{2020} Mathematics Subject Classification}
\makeatother

\subjclass[2020]%{\dots} 
{49N10, % Linear-quadratic optimal control problems
35R09; % Integro-partial differential equations
93C23, % (2000-now) Control/observation systems governed by functional-differential equations
49N35} %(1991-now) Optimal feedback synthesis


\begin{abstract}
We consider the linear quadratic (LQ) optimal control problem for a class of evolution equations in infinite dimensions, in the presence of distributed and nonlocal inputs.
Following the perspective taken in our previous research work on the LQ problem for integro-differential equations, where the memory term -- here involving the control actions -- is seen as a component of the state, we offer a full (closed-loop, Riccati-like) solution to the optimization problem.
\end{abstract}

\maketitle

% INTRO

\section{Introduction} \label{s:intro}
This article is devoted to the study of the quadratic optimal control problem on a finite time horizon for a class of linear systems in Hilbert spaces, in the presence of finite memory of the control actions, following the authors' recent work in \cite{ac-bu-memory1_2024} on the similar problem for evolution equations with memory (usuallly affecting the state variable, a feature brought about e.g. by the modeling of some diffusion processes or other physical phenomena that exhibit hereditary effects).
% such as viscoelasticity.

The linear system under consideration, that is \eqref{e:system-start}, covers evolutionary partial differential equations (PDE) with distributed control actions (both local and nonlocal ones).
We find this problem interesting enough in itself; in addition, the obtained results might also serve as an intermediate step towards future developments.
The question that is addressed is whether an optimal synthesis is viable by way of solving suitable quadratic (operator) equations associated with the optimization problem.   
To wit, we seek a closed-loop representation of the optimal control involving operators
which can be uniquely determined, just like it happens in the memoryless case via Riccati equations.
% do solve uniquely a system of Riccati or Riccati-type equations.

A fairly general model equation with memory in Hilbert spaces has been studied in \cite{pritchard-you_1996} with a similar goal, not fully attained though.
Indeed, a % semi-causal 
representation formula for the optimal control is established therein, with the feedback operator
depending on an another operator which is shown to solve a Fredholm integral equation\footnote{In the authors' words, the said equation \textquote{\dots plays a role similar to that of the operator Riccati equation}.}.
Thus, a reference work for these authors is \cite{pandolfi-memory_2018}, where the linear quadratic  (LQ) problem for a simple integro-differential model in $\mathbb{R}^d$ is studied and the optimal synthesis via Riccati-type equations is actually obtained (for the first time, to the best of our knowledge).
An enhanced result has appeared recently in \cite{pandolfi-memory_arXiv2024}.
Recent work of these authors on the same integro-differential system yet in infinite dimensions \cite{ac-bu-memory1_2024} thus constitutes an improvement upon \cite{pandolfi-memory_2018} and paved the way to the approach we present here.      
 
We would like to mention that most studies on optimal control problems for PDE with memory pertain to more general frameworks than the LQ one. They involve, e.g., semilinear PDE and/or non-quadratic functionals, and hence are aimed at establishing the existence of {\em open-loop} solutions as well as at characterizing them via first (and possibly, second) order optimality conditions; see e.g. \cite{cannarsa-etal_2013} and \cite{casas-yong-memory_2023}.
We also emphasize that in the presence of an infinite memory, the celebrated ``history approach'' of C.~Dafermos has proved effective in attaining a reformulation of the nonlocal problem;
the equivalent coupled system satisfied by an augmented variable may constitute a starting point for
the exploration of distinct control-theoretic properties -- such as uniform stability 
% and controllability, in the literature 
--, as well as inverse problems.
It must be stressed that all works \cite{pritchard-you_1996}, \cite{pandolfi-memory_2018} and \cite{ac-bu-memory1_2024} study {\em finite} time horizon problems for control system with {\em finite} memory, in this connection.
 
While all the aforementioned works pertain to deterministic evolution equations, it is in the realm of stochastic equations, possibly in the presence of fractional derivatives, that the literature appears richer.
See e.g. \cite{wang-t_2018}, \cite{bonaccorsi-confortola_2020}, \cite{abi-etal_2021a,abi-etal_2021b}
\cite{han-etal_2023}, \cite{wang-h-etal_2023}, \cite{hamaguchi-wang_2024}.
It goes withut saying that our list of references is by no means exhaustive; it is intended to  
provide a background which is most pertinent to our current work.
A suggested monograph on the subject is \cite{pandolfi-book}, along with the references therein. 

% METHODS
In terms of methods, on the one side we adopt the line of 
argument referred to as dynamic programming approach, according to which the minimization problem is embedded into a family of similar optimization problems depending on certain parameters, the initial time $s\in [0,T)$ and the initial state (what is actually meant for {\em state} here will be made
clear below). %more precisely here.
The goal is then to prove that the infimum of the functional -- which is a function of the said parameters -- solves an appropriate differential equation, that is Bellman's equation in the
general case and reduces to an operator Riccati equation in the LQ case. 
Solving this equation is key to the {\em synthesis} of the optimal control, based upon its representation in {\em closed-loop} form. 

% Yet in contrast with the said approach,
On the other hand, the proof of the well-posedness of a certain system of quadratic (operator) equations associated with the present optimization problem will not constitute the starting point of our analysis. % as it does
Rather, we perform and adapt an established path pursued %plan carried out 
in the study of the LQ problem for memoryless control system describing PDE problems, which develops through a series of principal steps (see \cite{las-trig-redbooks}).
These deal with 
\begin{itemize}

\item
the existence of a unique minimizer (the {\em open-loop} optimal control) via a 
convex optimization argument;

\item
the optimality condition, which brings about 

\item
the definition of an operator $P(t)$ (in terms of the optimal evolution), 
which is a building block of the optimal cost and also enters a certain feedback formula; 
this is a candidate to be a solution to the Riccati equation; 

\item
the issue of existence for the Riccati equation -- by way of proving that $P(t)$ {\em does} solve it
-- and

\item
the issue of uniqueness for the Riccati equation, which is key to the unique determination of 
the (optimal cost) operator $P(t)$, thereby allowing the synthesis of the optimal {\em feedback} control, to wit, in {\em closed-loop} form.
\end{itemize}

Performing this broad plan entails new challenges, certainly in comparison to the memoryless case,
and also to the analysis carried out in \cite{ac-bu-memory1_2024}.
We remark that: 
(i) we take the perspective previously utilized in our earlier work \cite{ac-bu-memory1_2024}, following \cite{pandolfi-memory_2018}, wherein the history is seen as a (second) component of the state;
(ii) the solution formula corresponding to the control system \eqref{e:system-start} naturally accounts for the more involved computations at any step of the investigation.
(iii) Once again, differently from the memoryless case, it is not apparent that the operators $P_i$, $i\in \{0,1,2\}$, that are building-blocks of the optimal cost actually occur in a first feedback formula that follows from the optimality condition;
(iv) also, the very same transition properties possessed by the optimal pair demand a careful analysis and tailored calculations; see the Appendix.
(v) Proving that the triplet $(P_0,P_1,P_2)$ is the {\em unique} solution to a certain system of coupled (operator) equations is much harder owing to the dependence of the space $Y_t=H\times L^2(0,t;U)$ on $t$; however, the analyis here benefits from the insight gained through our previous work in \cite[Section 6]{ac-bu-memory1_2024}.

% SETUP SUBSECTION 

\subsection{The problem setup} %Choice of the state space}
To formalize the problem and to state the main results of this paper we need to provide
a description, albeit minimal, of the mathematical framework.
%the reader to the mathematical framework.
% 
Let $H$ and $U$ be two separable complex Hilbert spaces, $T>0$ be given, and set $\cU:=L^2(0,T;U)$.
Consider the control system
\begin{equation} \label{e:system-start}
w'=Aw+Bu+\int_0^t k(t-\sigma)Bu(\sigma)\,d\sigma\,, \qquad t\in (0,T]\,,
\end{equation}
supplemented with an initial condition $w(0)=w_0\in H$; the function $u(\cdot)$ -- having a role of a control action -- varies in $\cU$.
The operators $A$ and $B$ and the kernel $k$ are assumed to satisfy the following properties.

% ASSUMPTIONS
 
\begin{assumptions}[\bf Basic Assumptions] \label{a:ipo_0} 
Let $H$, $U$ be separable complex Hilbert spaces. 
\begin{itemize}
\item
The %closed 
linear operator $A\colon \cD(A)\subset H \to H$ is the infinitesimal generator of a strongly continuous semigroup $\{e^{tA}\}_{t\ge 0}$ on $H$;

\item 
$B\in \cL(U,H)$;

\item
$k\in L^2(0,T;\mathbb{R})$.
\end{itemize}

\end{assumptions} 

% REMARK

\begin{remarks}
\begin{rm}
(i) As it is well-known, the boundedness of the control operator $B$ covers the case of partial
differential equations (PDE) systems subject to control actions which are {\em distributed} inside
the (bounded) domain. 
The analysis of the interesting case where {\em boundary} control actions are present -- resulting in
{\em unbounded} control operators $B$ (in the functional-analytic representation) -- is left to
future investigation.
(ii) The assumption that the memory kernel $k$ is real valued can be relaxed to 
$k\in L^2(0,T;\cL(H))$, with the computations carried out still valid, provided the operator 
$k(\cdot)$ commutes with the semigroup $e^{\cdot A}$.
\end{rm}
\end{remarks}

Under the Assumptions~\ref{a:ipo_0}, the Cauchy problem associated with the evolution equation \eqref{e:system-start} admits a unique mild solution
\begin{equation} \label{e:mild-sln}
w(t)=e^{At} w_0 +\int_0^t e^{A(t-q)}Bu(q)\,dq +\int_0^t e^{A(t-q)} \int_0^q k(q-p)Bu(p)\,dp\,dq
\end{equation}
that corresponds to a given initial datum $w_0\in H$ at the initial time $t=0$, as well as to
a control action $u(\cdot)\in L^2(0,T;U)$. 

To the model equation \eqref{e:system-start} we associate the following quadratic functional over the preassigned time interval $[0,T]$:
\begin{equation} \label{e:cost}
J(u)=J_T(u,w_0)=\int_0^T \left(\|Cw(t)\|_H^2 + \|u(t)\|_U^2\right)dt\,, 
\end{equation}
where the weighting operator $C$ simply satisfies 
\begin{equation} \label{e:ipo_1}
C\in \cL(H)\,.
\end{equation}
The simplified notation $J(u)$ should be self-explanatory and will be used throughout. 

\smallskip
The optimal control problem is formulated in the usual classical way.

% OCP

\begin{problem}[\bf The optimal control problem] \label{p:problem-0}
Given $w_0\in H$, seek a control function $\hat{u}(\cdot)=\hat{u}(\cdot,0,w_0)$ which minimizes the functional \eqref{e:cost} overall $u\in L^2(0,T;U)$, where $w(\cdot)$ is the mild solution to \eqref{e:system-start} (given by \eqref{e:mild-sln}) corresponding to the control function $u(\cdot)$ and with initial datum $w_0$ (at time $0$). 
\end{problem}

% DYNAMIC PROGRAMMING

\medskip
Consistently with the dynamic programming approach, the initial time $s$ is then allowed to vary in the interval $[0,T)$. 
Then, the solution formula \eqref{e:mild-sln} is rewritten accordingly: the input $u(\cdot)$ now
belongs to $L^2(s,T;U)$, while a novel function $\eta(\cdot)\in L^2(0,s;U)$ enters the representation,
that is
\begin{equation} \label{e:mild-sln_s}
w(t)=e^{A(t-s)} w_0 +\big[\big(L_s+H_s)u\big](t)+ \cK_s \eta(t)\,,
\end{equation}
with the operators $L_s$, $H_s$ and $\cK_s$ %(acting on the respective spaces) 
defined by
\begin{subequations}
\begin{align}
& L_s u(t)\equiv(L_s u)(t)=\int_s^t e^{A(t-q)}Bu(q)\,dq
\\[1mm]
& \label{e:operator_Hs} H_s u(t)\equiv(H_s u)(t)=\int_s^t e^{A(t-q)} \int_s^q k(q-p)Bu(p)\,dp\,dq
\\[1mm]
& \cK_s \eta(t)\equiv(\cK_s \eta)(t)=\int_s^t e^{A(t-q)} \int_0^s k(q-p)B\eta(p)\,dp\,dq\,.
\nonumber 
\end{align}
\end{subequations}
An easy computation allows a neat rewriting of the term $\cK_s \eta(t)$, that is
\begin{equation} \label{e:operator-cK}
\cK_s \eta(t) =\int_0^s \Big[\int_s^t e^{A(t-q)} k(q-p)B\,dq\Big]\eta(p)\,dp
=:\int_0^s \lambda(t,p,s)\eta(p)\,dp\,,
\end{equation}
where we introduced the operator $\lambda(t,p,s)$ defined by
\begin{equation*}
\lambda(t,p,s):=\int_s^t e^{A(t-q)} k(q-p)B\,dq\,.
\end{equation*}

One key idea of our argument is the consideration of the memory (here of the control function) 
up to time $s$ as a component of the {\em state}, when $s>0$; it then makes sense 
to group the terms which depend on $w_0$ and $\eta(\cdot)$ in \eqref{e:mild-sln_s}.
Thus, with an initial state $X_0$ which equals $w_0$ in the case $s=0$ while it subsumes the initial datum $w_0$ at time $s\in (0,T)$ and a given and yet arbitrary function $\eta(\cdot)$ defined in the interval $(0,s)$, we can re-express the solution formula \eqref{e:mild-sln_s} as 
\begin{equation} \label{e:formula_1}
w(t)=E(t,s)X_0+ \big[\big(L_s+H_s)u\big](t)\,, 
\end{equation}
where we set
\begin{equation}  \label{e:operator-E}
E(t,s)X_0:=e^{A(t-s)} w_0+\cK_s \eta(t)\,, 
\qquad X_0=\begin{pmatrix}w_0
\\[1mm] 
\eta(\cdot)\end{pmatrix}\,.
\end{equation}
The product space
\begin{equation}\label{e:state_space}
Y_s:= \begin{cases}
H & s=0
\\
H\times L^2(0,s;U) & 0<s<T
\end{cases}
\end{equation}
is the actual {\em state space} for the dynamics \eqref{e:mild-sln_s}.

Thus, along with the solution $w(\cdot)$ in \eqref{e:formula_1} -- corresponding to an initial state $X_0\in Y_s$ at time $s$ and a control function $u\in L^2(s,T;U)$ -- we introduce the family of
functionals
\begin{equation} \label{e:cost_s}
J_s(u)=J_{T,s}(u,X_0)=\int_s^T \big(\|Cw(t)\|_H^2 + \|u(t)\|_U^2\big)\,dt\,;
\end{equation}
the relative optimal control problem is formulated in a natural way.

% PARAMETRIC OCP

\begin{problem}[\bf Parametric optimal control problem] \label{p:problem_s}
Given $X_0\in Y_s$, seek a control function $\hat{u}=\hat{u}(\cdot,s,X_0)$ which minimizes the functional \eqref{e:cost_s} overall $u\in L^2(s,T;U)$, where $w(\cdot)$ -- given by \eqref{e:mild-sln_s} -- is the solution to \eqref{e:system-start} corresponding to a control function $u(\cdot)$ and with initial datum $X_0$. 
\end{problem}

% MAIN RESULTS

\subsection{Main results}
For expository purposes, the various principal outcomes of this work are gathered in this introduction
within a unique result.
 
% THM 1
 
\begin{theorem} \label{t:main}
With reference to the optimal control problem \eqref{e:mild-sln_s}-\eqref{e:cost_s}, under
the Assumptions~\ref{a:ipo_0} and the hypothesis \eqref{e:ipo_1}, the following statements are valid for any $s\in [0,T)$.

\begin{enumerate}

\item[\bf S1.] 
For each $X_0\in Y_s$ there exists a unique optimal pair 
$(\hat{u}(\cdot,s,X_0),\hat{w}(\cdot,s,X_0))$
which satisfies 
\begin{equation*}
\hat{u}(\cdot,s,X_0)\in C([s,T],U)\,, \quad \hat{w}(\cdot,s,X_0)\in C([s,T],H)\,.
\end{equation*}

\smallskip

\item[\bf S2.] 
Given $t\in [s,T]$, the linear bounded operator $\Phi(t,s)\colon Y_s \longrightarrow Y_t$ defined by 
\begin{equation} \label{e:evolution-map} 
\Phi(t,s)X_0 :=
\begin{pmatrix}
\hat{w}(t,s,X_0)
\\
\hat{v}(\cdot)
\end{pmatrix}
\; \text{\small where} \; \hat{v}(\cdot)=
\begin{cases} 
\eta(\cdot) & \text{in $[0,s)$}
\\
\hat{u}(\cdot,s,X_0) & \text{in $[s,t]$}
\end{cases}
\end{equation}
is an evolution operator, namely, it satisfies the transition property
\begin{equation*}
\Phi(t,t)=I\,, \qquad 
\Phi(t,s)=\Phi(t,\tau)\Phi(\tau,s) \quad \textrm{for \ $s\le \tau\le t\le T$.}
\end{equation*}

\smallskip

\item[\bf S3.]
There exist three linear bounded operators, denoted by $P_0(s)$, $P_1(s,p)$, $P_2(s,p,q)$ -- 
defined in terms of the optimal evolution and of the data of the problem
(see the expressions \eqref{e:riccati-ops} and \eqref{e:riccati-ops_2})  
--, such that the optimal cost is given by  
\begin{equation*} %\label{e:optimal-cost_1}
\begin{split}
\qquad J_s(\hat{u}) &=\big(P_0(s) w_0,w_0\big)_H
+ 2\text{Re}\, \int_0^s \big(P_1(s,p)\eta(p),w_1\big)_H \,dp
\\[1mm]
& + \int_0^s \!\!\int_0^s \big(P_2(s,p,q)\eta(p),\eta(q)\big)_U\,dp\,dq
\equiv \big(P(s)X_0,X_0\big)_{Y_s}\,.
\end{split}
\end{equation*}
$P_0(s)$ and  $P_2(s,p,q)$ are self-adjoint and non-negative operators in the respective functional
spaces $H$ and $L^2(0,s;U)$; in addition, it holds 
\begin{equation*}
P_2(s,p,q)=P_2(s,q,p)\,.
\end{equation*}

\smallskip

\item[\bf S4.] 
The optimal control admits the following representation in closed-loop form:
\begin{equation*} %\label{e:feedback_2}
\begin{split}
\hat{u}(t,s,X_0)&=-\big[B^*P_0(t)+P_1(t,t)^*\big]\hat{w}(t,s,X_0)
\\
& \qquad -\int_0^t \big[B^*P_1(t,p)+P_2(t,p,t)\big]\theta(p)\,dp\,,
\end{split}
\end{equation*}
with $\theta(\cdot)=\hat{u}(\cdot,0,w_0)\big|_{[0,t]}$
and the operators $P_i$ are given by the formulas \eqref{e:riccati-ops_2} (originally, \eqref{e:riccati-ops}), $i\in \{0,1,2\}$. 

\smallskip

\item[\bf S5.] 
The operators $P_0(t)$, $P_1(t,p)$, $P_2(t,p,q)$ -- as from {\bf S3.} --
satisfy the following coupled system of equations, for every $t\in[0,T)$, $s,q\in [0,t]$, and for any $x,y\in \cD(A)$, $v,u\in U$:
\begin{equation} \label{e:DRE}
\hspace{3mm}
\begin{cases}
&\frac{d}{dt}\big(P_0(t)x,y\big)_H +\big(P_0(t)x,Ay\big)_H 
+ \big(Ax,P_0(t)y\big)_H + \big(C^*Cx,y\big)_H  
\\[1mm] 
&\myspace - \big([B^*P_0(t)+P_1(t,t)^*]x,[B^*P_0(t)+P_1(t,t)^*]y\big)_U=0 
\\[2mm]
&\frac{\partial}{\partial t} \big(P_1(t,p)v,y\big)_H  + \big(P_1(t,p)v,Ay\big)_H 
+ \big(k(t-p)Bv,P_0(t)y\big)_H 
\\[1mm] 
&\qquad\quad
- \big([B^*P_1(t,p)+P_2(t,p,t)]v,[B^*P_0(t)+P_1(t,t)^*]y\big)_U =0
\\[2mm]
& \frac{\partial}{\partial t}\big( P_2(t,p,q)u,v\big)_U
+\big(P_1(t,p)u,k(t-q)Bv\big)_H
+\big(k(t-p)Bu,P_1(t,q)v\big)_H
\\[1mm]
&\qquad\quad  -\big([B^*P_1(t,p)+P_2(t,p,t)]u,[B^*P_1(t,q)+P_2(t,q,t)]v\big)_U=0
\end{cases}
\end{equation}
with final conditions
\begin{equation} \label{e:final}
P_0(T)=0\,, \; P_1(T,p)=0\,, \; P_2(T,p,q)=0\,.
\end{equation} 

\smallskip

\item[\bf S6.] {\bf (Uniqueness)}
There exists a unique triplet $(P_0(t),P_1(t,p),P_2(t,p,q))$ that solves the coupled system \eqref{e:DRE} and fulfils the final conditions \eqref{e:final}, within the class of linear bounded operators (in the respective spaces), the former and the latter being self-adjoint and non-negative.  
\end{enumerate}

\end{theorem}

% OVERVIEW

\subsection{An overview}
The paper is organized as follows.
In Section \ref{s:feedback} we deal with the statements S1., S3. and S4. of Theorem~\ref{t:main}: 
the optimality condition allows to obtain a first feedback representation of the unique optimal control, that is \eqref{e:feedback_1}.
Next, we identify three operators $P_i$ ($i=0,1,2$) which are building blocks of the quadratic form representing the optimal cost.
% (in the chosen state space)
That the said operators actually occur in the representation formula \eqref{e:feedback_1} necessitates that novel expressions for the said operators are devised; this key (and technical) task is accomplished in Lemma~\ref{l:key}.
Lemma~\ref{l:key} then enables us to attain a definitive feedback representation of the optimal
solution, that is \eqref{e:feedback_2}.

Section~\ref{s:existence} is entirely devoted to the proof of statement S5., namely, to derive the coupled system of three differential equations satisfied by the operators $P_i$, $i=0,1,2$. 
Section~\ref{s:uniqueness} focuses on the property of uniqueness.

The proof of the instrumental transition properties satisfied by the optimal control and evolution (implying statement S2.) has been postponed to an Appendix.

% SECTION 2

\section{The unique optimal control. Towards a feedback formula} \label{s:feedback}
Our goal in this section is twofold. 
We aim at attaining a {\em closed-loop} representation of the optimal control, as well as at disclosing that the said representation formula involves certain operators which are among the very same building-blocks of the optimal cost.

More precisely, the optimality condition will bring about a first feedback representation of the unique optimal control $\hat{u}=\hat{u}(t,s,X_0)$ -- an {\em open-loop} one at the outset -- depending both on the optimal evolution $\hat{w}(t,s,X_0)$, pointwise in time, and also on the full history of the optimal control up to time $t$; see formula \eqref{e:feedback_1}. 
This preliminary will necessitate that certain transition properties of the optimal solution
and optimal evolution are ascertained, and this is done in the Appendix.

On another side, we will see that the optimal cost $J_s(\hat{u})$ is a quadratic form on the state space $Y_s$, involving three linear and bounded operators.
However, unlike the memoryless case, these operators -- defined by \eqref{e:riccati-ops} -- 
cannot be singled out readily from the representation \eqref{e:feedback_1} of the optimal control.
% initially obtained via the optimality condition.

The next step -- a crucial one from a technical point of view -- is then a clever rewriting of
the aforesaid operators, accomplished in Lemma~\ref{l:key}. 
This result is the key to establish a second and definitive feedback representation of $\hat{u}$.

% 1^ FEEDBACK FORMULA

\subsection{Optimality condition, a first feedback formula}
By inserting the expression \eqref{e:formula_1} of $w(t)=w(t,s,X_0)$ in the quadratic functional \eqref{e:cost_s}, we find readily
\begin{equation} \label{e:q-form}
J_s(u,X_0)= \big(M_s X_0,X_0\big)_{Y_s}
+ 2 \text{Re}\, \big(N_s X_0,u\big)_{L^2(s,T;U)}
+ \big(\Lambda_s u,u\big)_{L^2(s,T;U)}\,,
\end{equation}
where
\begin{equation} \label{e:maiuscoli}
\begin{split}
% M_s
\langle M_s X_0,X_0\rangle_{Y_s}
&:=\int_s^T \langle C^*C E(t,s)X_0,E(t,s)X_0\rangle_H\,dt
\\[1mm]
% N_s
\big[N_s X_0\big](\cdot)&:=\big[\big(L_s^*+H_s^*\big) C^*C E(\cdot,s)X_0\big](\cdot)
\\[1mm]
% Lambda_s
\Lambda_s & :=I+\big(L_s^*+H_s^*\big)C^*C\big(L_s+H_s\big)\,,
\end{split}
\end{equation}
where $I$ denotes the identity operator on $L^2(s,T;U)$.

A standard argument is invoked now, owing to the quadratic structure of the functional: since 
$C^*C\ge 0$, then $\Lambda_s\ge I$; namely, the cost functional is coercive in the space 
$L^2(s,T;U)$ of admissible controls, and hence there exists a unique optimal control $\hat{u}$
minimizing the cost \eqref{e:cost_s}.
Owing to \eqref{e:formula_1}, we find that the component $\hat{w}$ of the optimal state belongs to 
$C([s,T],H)$. 

The optimality condition
\begin{equation}\label{e:optimality_0}
\Lambda_s \hat{u}+N_s X_0=0
\end{equation}
yields, on the one side, 
\begin{equation} \label{e:from-optimality}
\hat{u}=-\Lambda_s^{-1} N_s X_0\,, \quad X_0=\begin{pmatrix}w_0
\\[1mm] 
\eta(\cdot)\end{pmatrix}\,, 
\quad \eta(\cdot)=\hat{u}(\cdot,0,X_0)\big|_{[0,s]}\,. 
\end{equation}
On the other side, recalling \eqref{e:maiuscoli} and rewriting explicitly $\Lambda_s$, 
we attain first 
\begin{equation} \label{e:start-for-transition}
\big[I+\big(L_s^*+H_s^*\big)C^*C\big(L_s+H_s\big)\big]\hat{u}
+ \big(L_s^*+H_s^*\big)C^*CE(\cdot,s)X_0=0\,,
\end{equation}
which combined with \eqref{e:formula_1} yields
\begin{equation*}
\begin{split}
\hat{u}(\cdot,s,X_0)&=- \big(L_s^*+H_s^*\big)C^*C \big[E(\cdot,\tau) X_0 
+ \big(L_s+H_s\big)\hat{u}(\cdot,s,X_0)\big]
\\
&=- \big(L_s^*+H_s^*\big)C^*C \hat{w}(\cdot,s,X_0)
\end{split}
\end{equation*}
which explicitly reads as 
\begin{equation} \label{e:feedback_0}
\begin{split}
\hat{u}(t,s,X_0)=
&= -\int_t^T B^*e^{A^*(\sigma-t)}C^*C\hat{w}(\sigma,s,X_0)\,d\sigma
\\
&\qquad -\int_t^T \int_t^q B^*e^{A^*(q-p)}k(p-t)\,dp\,C^*C\hat{w}(q,s,X_0)\,dq\,.
\end{split}
\end{equation}
This is a very first representation of the optimal control $\hat{u}(\cdot)$ in terms of the component $\hat{w}(\cdot)$ of the optimal state. 
The regularity in time of $\hat{u}$ is enhanced from $L^2(s,T:U)$ to $C([s,T],U)$, thus confirming
the statement S1. of Theorem~\ref{t:main}.

\begin{remark}
\begin{rm}
We note that the dependence of $\hat{u}(t,s,X_0)$ on $s$ is via $\hat{w}(\cdot,s,X_0)$ only.
\end{rm}
\end{remark}

% TOWARD the 1^ REPRESENTATION of the OPTIMAL CONTROL

Let us go back to the optimality condition \eqref{e:optimality_0}.
Recalling the definitions \eqref{e:operator-E} and \eqref{e:operator-cK} of the operators $E$ and 
$\cK$, respectively, we see that 
\begin{equation} \label{e:opt-control}
\begin{split}
\hat{u}(t,s,X_0)&=-\Lambda_s^{-1} N_s X_0
=-\Lambda_s^{-1} \big(L_s^*+H_s^*\big) C^*C E(\cdot,s)X_0
\\
&= -\Lambda_s^{-1} \big(L_s^*+H_s^*\big) C^*C e^{A(\cdot-s)}w_0
-\Lambda_s^{-1} \big(L_s^*+H_s^*\big) C^*C \int_0^s \lambda(t,p,s)\eta(p)\,dp
\\
&=: \psi_1(t,s)w_0+ \int_0^s \psi_2(t,p,s)\eta(p)\,dp\,,
\end{split}
\end{equation}
where we have set
\begin{equation} \label{e:psi_i}
\begin{split}
\psi_1(t,s)&:= -\Lambda_s^{-1} \big[\big(L_s^*+H_s^*\big) C^*C e^{A(\cdot-s)}\big](t)\,,
\\[1mm]
\psi_2(t,p,s)&:= -\big[\Lambda_s^{-1} \big(L_s^*+H_s^*\big) C^*C \lambda(\cdot,p,s)\big](t)\,.
\end{split}
\end{equation}
Inserting the expression \eqref{e:opt-control} of $\hat{u}$ in \eqref{e:formula_1}, 
we get
\begin{equation} \label{e:opt-state-component}
\begin{split}
\hat{w}(t,s,X_0)&=e^{(t-s)A} w_0+ \int_0^s \lambda(t,p,s)\eta(p)\,dp 
+\big(L_s+H_s)\hat{u}(t)
\\
& = e^{(t-s)A} w_0+ \int_0^s \lambda(t,p,s)\eta(p)\,dp 
\\
& \qquad + \int_s^t e^{A(t-\sigma)}B\Big[\psi_1(\sigma,s)w_0
+ \int_0^s \psi_2(\sigma,p,s)\eta(p)\,dp\Big]d\sigma
\\
& \qquad 
+ \int_s^t \int_\sigma^t e^{A(t-q)}k(q-\sigma)\,dq 
B\Big[\psi_1(\sigma,s)w_1+ \int_0^s \psi_2(\sigma,p,s)\eta(p)\,dp\Big]d\sigma 
\\
&=: Z_1(t,s)w_0+ \int_0^s Z_2(t,p,s)\eta(p)\,dp\,,
\end{split}
\end{equation}
having set 
\begin{equation} \label{e:Z_i}
\begin{split}
Z_1(t,s)&:= e^{A(t-s)}+\int_s^t e^{A(t-\sigma)}B\psi_1(\sigma,s)\,d\sigma
+ \int_s^t \int_\sigma^t e^{A(t-q)}k(q-\sigma)\,dq B\psi_1(\sigma,s)\,d\sigma
\\[1mm]
& % new 
\equiv e^{A(t-s)}+\big[\big(L_s+H_s\big)\psi_1(\cdot,s)\big](t)\,,
\\[2mm]
Z_2(t,p,s)&:= \lambda(t,p,s)+ \int_s^t e^{A(t-\sigma)}B\psi_2(\sigma,p,s)\,d\sigma
\\
& \qquad
+ \int_s^t \int_\sigma^t e^{A(t-q)}k(q-\sigma)\,dq B\psi_2(\sigma,p,s)\,d\sigma\,.
\end{split}
\end{equation}

We now return to the very first representation \eqref{e:feedback_0} of the optimal control.
Set
\begin{equation*}
X_1=\begin{pmatrix}w_1\\[1mm] \theta(\cdot)\end{pmatrix}, 
\;  w_1=\hat{w}(t,s,X_0)\,, \; \theta(\cdot)=\hat{u}(\cdot,0,w_0)\big|_{[0,t]}
\end{equation*}
% CHECK: credo di dover introdurre una proposizione esplicita, da richiamare
and use the transition property of the optimal pair (whose proof is deferred to the Appendix for conciseness) to find
\begin{equation*}
\begin{split}
\hat{u}(t,s,X_0)&= -\int_t^T B^*e^{A^*(\sigma-t)}C^*C \hat{w}(\sigma,t,X_1)\,d\sigma
\\
& \quad
-\int_t^T \int_t^q B^*e^{A^*(q-p)}k(p-t)^*\,dp\,C^*C \hat{w}(q,t,X_1)\,dq
\\
& = -\int_t^T B^*e^{A^*(\sigma-t)}C^*C Z_1(\sigma,t)w_1\,d\sigma
\\
& \quad
-\int_t^T B^*e^{A^*(\sigma-t)}C^*C  \int_0^t Z_2(\sigma,p,t)\theta(p)\,dp\,d\sigma
\\
& \quad
- \int_t^T \int_t^q B^*e^{A^*(q-p)}k(p-t)^*\,dp\, C^*C Z_1(q,t)w_1\,dq
\\
& \quad
-\int_t^T \int_t^q B^*e^{A^*(q-p)}k(p-t)^*\,dp \,C^*C \int_0^t Z_2(q,r,t)\theta(r)\,dr\,dq
\end{split}
\end{equation*}
which brings about the first representation formula
\begin{equation} \label{e:feedback_1}
\begin{split}
\hat{u}(t,s,X_0)
& = -\int_t^T B^*e^{A^*(\sigma-t)}C^*C Z_1(\sigma,t)w_1\,d\sigma
\\
& \quad
- \int_t^T \int_t^q B^*e^{A^*(q-p)}k(p-t)^*\,dp\, C^*C Z_1(q,t)w_1\,dq
\\
& \quad
-\int_t^T B^*e^{A^*(\sigma-t)}\,C^*C \int_0^t Z_2(\sigma,p,t)\theta(p)\,dp\,d\sigma
\\
% & \quad -\int_t^T \underbrace{\int_t^q B^*e^{A^*(q-p)}k(p-t)^*\,dp}_{\lambda(q,t,t)^*}\,
% C^*C \int_0^t Z_2(q,r,t)\theta(r)\,dr\,dq\,.
 & \quad -\int_t^T \lambda(q,t,t)^*\,C^*C \int_0^t Z_2(q,r,t)\theta(r)\,dr\,dq\,.
\end{split}
\end{equation}
We note that although $\hat{u}(t,s,X_0)$ might seem to be independent of $s$, this is not the case: the dependence on $s$ is via $w_1$.
 
% 3 RICCATI operators 
 
\subsection{The optimal cost operators}
We now evaluate the quadratic functional $J_s(u)$ on the optimal control $\hat{u}$, by inserting the expressions \eqref{e:opt-control} and \eqref{e:opt-state-component} of the optimal pair (comprising the optimal control and the first component of the optimal state) in it, to find
\begin{equation*} 
\begin{split}
J_s(\hat{u})&=J_s(\hat{u},X_0)
= \int_s^T \Big[\|CZ_1(t,s)w_0\|_H^2+ \big\|C\int_0^s Z_2(t,p,s)\eta(p)\, dp\big\|_H^2\Big]\,dt
\\
& \quad + 2 \text{Re}\int_s^T \big(CZ_1(t,s)w_0,C\int_0^s Z_2(t,p,s)\eta(p)\, dp\big)_H\,dt
\\
& \quad + \int_s^T \Big[\|\psi_1(t,s)w_1\|_U^2+ \big\|\int_0^s \psi_2(t,p,s)\eta(p)\,dp\big\|_U^2\Big]\,dt
\\
& \quad + 2 \text{Re}\int_s^T \big(\psi_1(t,s)w_0,\int_0^s \psi_2(t,p,s)\eta(p)\,dp\big)_U\,dt\,,
\end{split}
\end{equation*}
which is rewritten as
\begin{equation} \label{e:J-expression}
\begin{split}
J_s(u)
& = \int_s^T \Big[\big(Z_1(t,s)^*C^*C Z_1(t,s) w_0,w_0\big)_H
+\big(\psi_1(t,s)^*\psi_1(t,s) w_0,w_0\big)_U\Big]\,dt
\\
& \quad + 2 \text{Re}\Big[\int_s^T \Big(\int_0^s Z_1(t,s)^*C^*C Z_2(t,p,s)\eta(p)\,dp,w_0\Big)_H\,dt
\\
& \qquad + \int_s^T \Big(\int_0^s \psi_1(t,s)^*\psi_2(t,p,s)\eta(p)\,dp,w_0\Big)_H\,dt\Big]
\\
& \quad + \int_s^T \Big[\|\psi_1(t,s)w_1\|_U^2+ \big\|\int_0^s \psi_2(t,p,s)\eta(p)\,dp\big\|_U^2\Big]\,dt
\\
& \quad + \int_s^T \int_0^s\int_0^s \big(Z_2(t,q,s)^*C^*C Z_2(t,p,s)\eta(p),\eta(q)\big)_U\,dp\,dq\,dt
\\
& \quad + \int_s^T \int_0^s\int_0^s \big(\psi_2(t,q,s)^*\psi_2(t,p,s)\eta(p),\eta(q)\big)_U\,dp\,dq\,dt\,.
\end{split}
\end{equation}
Thus, by introducing the operators $P_0(s)\in  \cL(H)$, $P_1(t,s)\in \cL(U,H)$, $P_2(s,p,q)\in \cL(U)$ defined by
\begin{subequations} \label{e:riccati-ops}
\begin{align}
P_0(s)&=\int_s^T \big[Z_1(t,s)^*C^*C Z_1(t,s)+\psi_1(t,s)^*\psi_1(t,s)\big]\,dt\,,
\label{e:P_0_v1}
\\[1mm]
P_1(s,p)&=\int_s^T \big[Z_1(t,s)^*C^*C Z_2(t,p,s)+\psi_1(t,s)^*\psi_2(t,p,s)\big]\,dt\,,
\label{e:P_1_v1}
\\[1mm]
P_2(s,p,q)&= \int_s^T \big[Z_2(t,q,s)^*C^*C Z_2(t,p,s)+\psi_2(t,q,s)^*\psi_2(t,p,s)\big]\,dt\,,
\label{e:P_2_v1}
\end{align}
\end{subequations} 
it is clear from \eqref{e:J-expression} that the optimal cost reads as the following quadratic
form on $Y_s$:
\begin{equation} \label{e:form}
\begin{split}
J_s(\hat{u})=J_s(\hat{u},X_0)
&=\big(P_0(s) w_0,w_0\big)_H
+ 2 \text{Re}\Big(\int_0^s P_1(s,p)\eta(p)\,dp,w_0\Big)_H
\\
& \myspace + \int_0^s\int_0^s \big(P_2(s,p,q)\eta(p),\eta(q)\big)_U\,dp\,dq\,.
\end{split}
\end{equation} 

\subsection{Second feedback formula}
The presence of the operators $P_0$, $P_1$ and $P_2$ (defined by \eqref{e:riccati-ops} and which 
are building-blocks of the quadratic form \eqref{e:form}) %optimal cost
is not immediately apparent in the obtained formula \eqref{e:feedback_1} for the optimal strategy.
Pinpointing this fact requires that we deduce a suitable distinct representation of each operator,
beforehand.

% KEY LEMMA

\begin{lemma} \label{l:key}
With the operators $\psi_1(t,s)$ and $\psi_2(t,p,s)$ defined in \eqref{e:psi_i}, $Z_1(t,s)$ and $Z_2(t,p,s)$ defined in \eqref{e:Z_i}, then the optimal cost operators $P_i$ in \eqref{e:riccati-ops} can
be equivalently rewritten as follows, $i\in \{0,1,2\}$:
\begin{subequations} \label{e:riccati-ops_2}
\begin{align}
P_0(t)&=\int_t^T e^{A^*(\sigma-t)}C^*CZ_1(\sigma,t)\,d\sigma\,,
\label{e:P_0_v2}
\\[1mm]
P_1(t,p)&=\int_t^T e^{A^*(\sigma-t)}C^*C Z_2(\sigma,p,t)\,d\sigma\,,
\label{e:P_1_v2}
\\[1mm]
P_2(t,p,q)&=\int_t^T \lambda(\sigma,q,t)^*C^*C Z_2(\sigma,p,t)\,dp\,.
\label{e:P_2_v2}
\end{align}
\end{subequations}
Moreover, an additional (third) expression of $P_1(t,p)$ holds true:
\begin{equation} \label{e:P_1_v3}
P_1(t,p)=\int_t^T  Z_1(\sigma,t)^*\,C^*C\lambda(\sigma,p,t)\,d\sigma\,,
\end{equation}
so that in particular
\begin{equation} \label{e:P_1^*(t,t)}
P_1(t,t)^* = \int_t^T \lambda(\sigma,t,t)^*\,C^*C Z_1(\sigma,t)\,d\sigma\,.
\end{equation}

\end{lemma}

\begin{proof}
{\bf 1.} We start from the definition \eqref{e:P_0_v1} of $P_0$:
recall the respective expressions for the operators $\psi_i$ and $Z_i$, $i=1,2$, to infer first
\begin{equation*}
\begin{split}
P_0(t)&=\int_t^T \Big\{\big[e^{A^*(\cdot-t)}+\psi_1(\cdot,t)^*\big(L_t^*+H_t^*\big)\big]
C^*C \big[e^{A(\cdot-t)}+\big(L_t+H_t\big)\psi_1(\cdot,t)\big]
\\
& \quad 
+\big[\psi_1(\sigma,t)^*\psi_1(\sigma,t)\big]\Big\}\,d\sigma\,
\\
&=\int_t^T e^{A^*(\sigma-t)}C^*C e^{A(\sigma-t)}\,d\sigma
+ \int_t^T e^{A^*(\sigma-t)}\big(L_t^*+H_t^*\big)\psi_1(\sigma,t)\,d\sigma
\\
& \quad
+ \int_t^T \Big\{\psi_1(\cdot,t)^*\big[\big(L_t^*+H_t^*\big)C^*C e^{A(\cdot-t)}\big]
\\
& \qquad + \psi_1(\cdot,t)^*\big(L_t^*+H_t^*\big)C^*C \big(L_t+H_t\big)\psi_1(\cdot,t) +\psi_1(\cdot,t)^*\psi_1(\cdot,t)\Big\}\,d\sigma\,.
\end{split}
\end{equation*}
Then
\begin{equation*}
\begin{split}
P_0(t)&=
\int_t^T e^{A^*(\sigma-t)}C^*C e^{A(\sigma-t)}\,d\sigma
+ \int_t^T e^{A^*(\sigma-t)}\big[\big(L_t^*+H_t^*\big)\psi_1(\cdot,t)\big](\sigma)\,d\sigma
\\
& \quad
+ \int_t^T \psi_1(\sigma,t)^*\cancel{\Big[\Lambda_t \psi_1(\cdot,t) - \Lambda_t \psi_1(\cdot,t)\Big]}(\sigma)\,d\sigma 
\\
& = \int_t^T e^{A^*(\sigma-t)}C^*C
\Big[e^{A(\sigma-t)}+ \big[\big(L_t^*+H_t^*\big)\psi_1(\cdot,t \big](\sigma)\Big]\,d\sigma
\\
& = \int_t^T e^{A^*(\sigma-t)}C^*CZ_1(\sigma,t)\,d\sigma\,,
\end{split}
\end{equation*}
which confirms \eqref{e:P_0_v2}.

% 2.
\smallskip
{\bf 2.} Next, starting again from \eqref{e:P_1_v1} we compute
\begin{equation*}
\begin{split}
P_1(t,p)&=\int_t^T \Big[Z_1(\sigma,t)^*C^*CZ_2(\sigma,p,t)
+\psi_1(\sigma,t)^*\psi_2(\sigma,p,t)\Big]\,d\sigma
\\
&= \int_t^T \Big\{\big[e^{A^*(\sigma-t)}+\psi_1(\sigma,t)^*\big(L_t^*+H_t^*\big)\big]
C^*C \big[\lambda(\cdot,p,t)+\big(L_t+H_t\big)\psi_2(\cdot,p,t)\big](\sigma)
\\
& \myspace 
+\big[\psi_1(\sigma,t)^*\psi_2(\sigma,p,t)\big]\Big\}\,d\sigma\,
\\
&=\int_t^T e^{A^*(\sigma-t)}C^*C \lambda(\sigma,p,t)\,d\sigma
+ \int_t^T e^{A^*(\sigma-t)}C^*C\big(L_t^*+H_t^*\big)\psi_2(\cdot,p,t)\,d\sigma
\\
& \quad
+ \int_t^T \psi_1(\cdot,t)^*\big[\big(L_t^*+H_t^*\big)C^*C \lambda(\cdot,p,t)\big](\sigma)\,d\sigma
\\
& \qquad
+ \int_t^T \psi_1(\cdot,t)^*\Big[\big(L_t^*+H_t^*\big)C^*C \big(L_t+H_t\big)\psi_2(\cdot,p,t)\Big](\sigma)\,d\sigma
\\
& \myspace +\int_t^T\psi_1(\sigma,t)^*\psi_2(\sigma,p,t)\,d\sigma
\end{split}
\end{equation*}
to find
\begin{equation*}
\begin{split}
P_1(t,p)&=\int_t^T e^{A^*(\sigma-t)}C^*C \Big[\lambda(\sigma,p,t)
+\big[\big(L_t^*+H_t^*\big)\psi_2(\cdot,p,t)\big](\sigma)\Big]\,d\sigma
\\
& \qquad +\int_t^T \psi_1(\cdot,t)^* 
\cancel{\Big[ -\Lambda_t\psi_2(\cdot,p,t) + \Lambda_t\psi_2(\cdot,p,t)\Big]}\,d\sigma
\\
& = \int_t^T e^{A^*(\sigma-t)}C^*C Z_2(\sigma,p,t)\,d\sigma\,,
\end{split}
\end{equation*}
thus establishing \eqref{e:P_1_v2}.

\smallskip
Since it will be crucial for the purpose of attaining a more eloquent representation formula for the optimal control than the former \eqref{e:feedback_1}, we compute the ajoint operator of $P_1(t,p)$
explicitly. 
% in two different ways (??) thereby attaining two equivalent expressions.
%
With \eqref{e:P_1_v2} as a starting point, we substitute the expression of $Z_2(\sigma,p,t)^*$
first, to find
\begin{equation*}
\begin{split}
P_1(t,p)^*& =\int_t^T Z_2(\sigma,p,t)^* C^*C e^{A(\sigma-t)}\,d\sigma
\\
& = \int_t^T \lambda(\sigma,p,t)^*\, C^*C e^{A(\sigma-t)}\,d\sigma
+ \int_t^T\psi_2(\sigma,p,t)^*\big[\big(L_t^*+H_t^*\big) C^*C e^{A(\cdot-t)}\big](\sigma)\,d\sigma\,.
\end{split}
\end{equation*}
% new
Next, we utlize again the expressions of $Z_1(\sigma,t)$ and $\psi_1(\cdot,t)$ in \eqref{e:Z_i} and \eqref{e:psi_i}, respectively, to rewrite as follows:
\begin{equation*}
\begin{split}
P_1(t,p)^*&=\int_t^T \lambda(\sigma,p,t)^*\,C^*C \Big[Z_1(\sigma,t)
- \big[\big(L_t+H_t\big)\psi_1(\cdot,t)\big](\sigma)\Big]\,d\sigma
\\
& \qquad 
+ \int_t^T\psi_2(\sigma,p,t)^*\big[-\Lambda_t\psi_1(\cdot,t)\big](\sigma)\,d\sigma
\\
&= \int_t^T \lambda(\sigma,p,t)^*\,C^*C Z_1(\sigma,t)\,d\sigma
- \int_t^T \lambda(\sigma,p,t)^*\,C^*C \big[\big(L_t+H_t\big)\psi_1(\cdot,t)\big](\sigma)\Big]\,d\sigma
\\
& \qquad 
- \int_t^T\psi_2(\sigma,p,t)^*\big[\Lambda_t\psi_1(\cdot,t)\big](\sigma)\,d\sigma\,,
\end{split}
\end{equation*}
which yields
\begin{equation} \label{e:P_1*}
\begin{split}
P_1(t,p)^*&= \int_t^T \lambda(\sigma,p,t)^*\,C^*C Z_1(\sigma,t)\,d\sigma
- \int_t^T \lambda(\sigma,p,t)^*\,C^*C \big[\big(L_t+H_t\big)\psi_1(\cdot,t)\big](\sigma)\,d\sigma
\\
&\; \qquad - \int_t^T \Big[\big[-\lambda(\sigma,p,t)^*\,C^*C \big(L_t+H_t\big)\Lambda_t^{-1}\big]\Lambda_t\psi_1(\cdot,t)\Big](\sigma)\,d\sigma
\\
& = \int_t^T \lambda(\sigma,p,t)^*\,C^*C Z_1(\sigma,t)\,d\sigma\,,
\end{split}
\end{equation}
as the second and third terms cancel.
The representation \eqref{e:P_1*} of $P_1(t,p)^*$ proves \eqref{e:P_1^*(t,t)}, while it establishes as well the (third) representation \eqref{e:P_1_v3} of $P_1(t,p)$.

% 3.

\smallskip
{\bf 3.} It remains to give a convenient representation of the operator $P_2(s,p,q)$.
We have
\begin{equation*}
\begin{split}
P_2(t,p,t)&=\int_t^T \Big[Z_1(\sigma,t,t)^*C^*CZ_2(\sigma,p,t)
+\psi_2(\sigma,t,t)^*\psi_2(\sigma,p,t)\Big]\,d\sigma
\\
&= \int_t^T \Big[\big[\lambda(\sigma,t,t)^*+\psi_2(\cdot,t,t)^*\big(L_t^*+H_t^*\big)\big]
C^*C \big[\lambda(\cdot,p,t)+\big(L_t+H_t\big)\psi_2(\cdot,p,t)\big](\sigma)
\\
& \myspace 
+\big[\psi_2(\sigma,t,t)^*\psi_2(\sigma,p,t)\Big]\,d\sigma
\\
&= \int_t^T \Big[\lambda(\sigma,t,t)^*\,C^*C\lambda(\sigma,p,t)
+\big[\lambda(\sigma,t,t)^*\,C^*C\big(L_t+H_t\big)\psi_2(\cdot,p,t)\big](\sigma)\Big]\,d\sigma
\\
& \qquad 
+ \int_t^T \Big[\psi_2(\sigma,t,t)^*\big[\big(L_t^*+H_t^*\big)\,C^*C\lambda(\cdot,p,t)\big](\sigma)
+ \psi_2(\sigma,t,t)^*\big[\big(L_t^*+H_t^*\big)\,C^*C\big(L_t+H_t\big)\psi_2(\cdot,p,t)\Big](\sigma)
\\
& \myspace + \psi_2(\sigma,t,t)^*\psi_2(\sigma,p,t)\Big]\,d\sigma
\\
& = \int_t^T \lambda(\sigma,t,t)^*\,C^*C Z_2(\sigma,p,t)\,d\sigma
+ \int_t^T \psi_2(\sigma,t,t)^* 
\cancel{\Big[-\Lambda_t\psi_2(\cdot,p,t)+\Lambda_t\psi_2(\cdot,p,t)\Big]}\,d\sigma
\\
& = \int_t^T \lambda(\sigma,t,t)^*\,C^*C Z_2(\sigma,p,t)\,d\sigma
\end{split}
\end{equation*}
which confirms the validity of \eqref{e:P_2_v2}. 

\end{proof}

% FEEDBACK FORMULA

Taking into account Lemma~\ref{l:key}, the representation formula \eqref{e:feedback_1}
is interpreted in a more eloquent way.

\begin{proposition} \label{p:feedback-formula}
Let $\hat{u}(t,s;X_0)$ be the optimal control for the minimization problem \eqref{e:mild-sln_s}-\eqref{e:cost_s}, with initial state $X_0$.
Then, the representation \eqref{e:feedback_1} of the optimal control $\hat{u}$ reads also as
\begin{equation} \label{e:feedback_2}
\begin{split}
\hat{u}(t,s,X_0)&=-\big[B^*P_0(t)+P_1(t,t)^*\big]\hat{w}(t,s,X_0)
\\
& \myspace -\int_0^t \big[B^*P_1(t,p)+P_2(t,p,t)\big]\theta(p)\,dp\,, 
\end{split}
\end{equation}
with % $w_2=\hat{w}(t,s,X_1)$ and 
$\theta(\cdot)=\hat{u}(\cdot,0,w_0)\big|_{[0,t]}$
and the operators $P_i$ as in \eqref{e:riccati-ops_2} (originally, in \eqref{e:riccati-ops}), 
$i\in \{0,1,2\}$. 

\end{proposition}

% SECTION 3: the DRE, SYSTEM of 3 eq.

\section{Proof of statement S5: existence for a corresponding system of three coupled equation}
\label{s:existence}
In this section we show that the triplet $(P_0(t),P_1(t,p),P_2(t,p,q))$ constituted by the operators which are building blocks for the optimal cost \eqref{e:form} and enter
the feedback representation \eqref{e:feedback_2} solves the system of coupled (operator)
equations \eqref{e:DRE}.
We point out here that this fact establishes the {\em existence} for a quadratic (Riccati-type) equation corresponding to the optimal control problem on $[s,T]$, with $Y_s$ as state space, which subsumes the said system \eqref{e:DRE}; see \eqref{e:big-DRE}.
% 
% the original system of three coupled (operator) equations in the unknowns $P_0(t)$, $P_1(t,p)$, 
% $P_2(t,p,q)$ will be completed with the equation satisfied by $P_1(t,p)^*$ and recast
% as a quadratic Riccati equation in $Y_s$.

A preliminary step which is unavoidable is the computation of certain derivatives of the operators $Z_1(p,\sigma)$ and $Z_2(\sigma,p,t)$, in turn based on differentiability results pertaining to the operators $\psi_1(p,t)$ and $\psi_2(r,p,t)$, respectively.   
This is provided in the following result, whose prove is omitted for the sake of brevity.
% postponed to Appendix~\ref{a:app_b}.
%
(We just leave the following tip: in order to establish \eqref{e:D-psi_1}, we go back to the definitions \eqref{e:psi_i}, write down the equation satisfied by the incremental ratio 
$(\psi_1(p,t+h)-\psi_1(p,t))/h$, $h>0$, and eventually take the limit, as $h\to 0$.
Formula \eqref{e:D-psi_2} is attained proceeding similarly, {\em mutatis mutandis}.
Then, the respective formulas \eqref{e:D-Z_1} and \eqref{e:D-Z_2} are readily obtained.)
 
% LEMMA: DERIVATIVES of psi_1, psi_2, Z_1, Z_2

\begin{lemma} \label{l:derivatives}
Let $\psi_1(p,t)$, $\psi_2(r,p,t)$, $Z_1(p,\sigma)$, $Z_2(\sigma,p,t)$ as from \eqref{e:psi_i}
and \eqref{e:Z_i}, respectively.
If $x\in \cD(A)$ and $v\in U$, then the derivatives 
$\partial_t\psi_1(p,t)x$, $\partial_t\psi_2(r,p,t)v$, 
$\partial_t Z_1(p,\sigma)x$, $\partial_t Z_2(\sigma,p,t)v$ exist, with

% FORMULE

\begin{align} %\label{e:derivatives}
% Z_1
\partial_t Z_1(\sigma,t)x 
&= -e^{(\sigma-t)}Ax +\big[\big(L_t+H_t\big)\partial_t \psi_1(\cdot,t)x\big](\sigma)
\notag\\
& \qquad
- \big[e^{(\sigma-t)A}B+\lambda(\sigma,t,t)\big]\psi_1(t,t)x\,,
\label{e:D-Z_1}
\\[2mm]
% Z_2
\partial_t Z_2(\sigma,p,t)v &= 
-e^{(\sigma-t)A}k(t-p)Bv +\big[\big(L_t+H_t\big)\partial_t \psi_2(\cdot,p,t)v\big](\sigma)
\notag\\
& \qquad- \big[e^{(\sigma-t)A}B+\lambda(\sigma,p,t)\big]\psi_2(\sigma,p,t)v\,,
\label{e:D-Z_2}
\end{align}
where
\begin{align} 
\partial_t \psi_1(p,t)x &= 
\Lambda_t^{-1}\Big[\big(L_t^*+H_t^*\big)C^*C\big[\big(e^{(\cdot-t)A}B+\lambda(\cdot,t,t)\big)\big]
\psi_1(t,t)x
\notag\\
& \qquad + e^{(\cdot-t)A}Ax\Big](p)\,,
\label{e:D-psi_1}
\\[2mm]
\partial_t \psi_2(r,p,t)v &= 
\Lambda_t^{-1}\Big[\big(L_t^*+H_t^*\big)C^*C\big[\big(e^{(\cdot-t)A}B+\lambda(\cdot,t,t)\big)\big]\psi_2(t,p,t)v
\notag \\
& \qquad + e^{(\sigma-t)A}k(t-p)Bv\Big](r)\,.
\label{e:D-psi_2}
\end{align}

% END LEMMA
\end{lemma}

% PROOF of S5

\medskip
\noindent
{\em Proof of statement S5.}
% eq. #1 for P_0
{\bf 1.} Let $x,y\in \cD(A)$.
Recall the expression \eqref{e:P_0_v2} of $P_0(t)$ and exploit Lemma~\ref{l:derivatives} to compute  
\begin{align*}
& \frac{d}{dt}\big(P_0(t)x,y\big)_H
=-\big(C^*CZ_1(t,t)x,y\big)_H-\int_t^T \big( C^*CZ_1(\sigma,t)x,e^{(\sigma-t)A}Ay\big)_H\,d\sigma
\\
& \quad +\int_t^T \big(C^*C \partial_t Z_1(\sigma,t)x,e^{(\sigma-t)A}y\big)_H\,d\sigma
\\
&=-\big(C^*Cx,y\big)-\big(P_0(t)x,Ay\big) + \int_t^T \big(e^{A^*(\sigma-t)}C^*C[-e^{(\sigma-t)A}Ax],y\big)_H\,d\sigma
\\
& \quad + \int_t^T \big(e^{(\sigma-t)A^*}C^*C 
\big[\big(L_t+H_t)\Lambda_t^{-1}\big(L_t^*+H_t^*\big)_H
\\
& \myspace C^*C\big[e^{(\cdot-t)A}B+\lambda(\cdot,t,t)\big]\psi_1(t,t)x
+ e^{(\cdot-t)A}Ax\big](\sigma),y\big)_H\,d\sigma
\\
& \quad - \int_t^T \big(e^{(\sigma-t)A^*}C^*C[e^{(\sigma-t)A}B+\lambda(\sigma,t,t)]
\psi_1(t,t)x,y\big)_H\,d\sigma\,.
\end{align*}
We single out the terms which display the elements $-e^{(\sigma-t)A}Ax$, $e^{(\sigma-t)A}Bx$,
$\lambda(\sigma,t,t)$, merge them accordingly to rewrite
\begin{align*}
& \frac{d}{dt}\big(P_0(t)x,y\big)_H 
= -\big(C^*C x,y\big)-\big(P_0(t)x,Ay\big)_H\,
\\
& \qquad -\int_t^T \big(e^{(\sigma-t)A^*}C^*C
\big\{e^{(\sigma-t)A}+\big[\big(L_t+H_t)\psi_1(\cdot,t)\big](\sigma)\big\}Ax,y\big)_H\,d\sigma
\\
&\qquad -\int_t^T \big(e^{(\sigma-t)A^*} C^*C
\big\{e^{(\sigma-t)A}+\big[\big(L_t+H_t)\psi_1(\cdot,t)\big](\sigma)\big\}B\psi_1(t,t)x,y\big)_H\,d\sigma
\\
&\qquad -\int_t^T \big(e^{(\sigma-t)A^*} C^*C
\big\{\lambda(\sigma,t,t)+\big[\big(L_t+H_t)\psi_2(\cdot,t,t)\big](\sigma)\big\}\psi_1(t,t)x,y\big)_H\,d\sigma
\\[1mm]
& = -\big(C^*C x,y\big)_H-\big(P_0(t)x,Ay\big)_H\,d\sigma
\\
& \quad -\int_t^T \big(e^{(\sigma-t)A^*}C^*C Z_1(\sigma,t)Ax,y\big)_H\,d\sigma
\\
& \quad -\int_t^T \big(e^{(\sigma-t)A^*}C^*C Z_1(\sigma,t)B\psi_1(t,t)x,y\big)_H\,d\sigma
\\
& \quad -\int_t^T \big(e^{(\sigma-t)A^*}C^*C Z_2(\sigma,t,t)\psi_1(t,t)x,y\big)_H\,d\sigma\,.
\end{align*}
that is
\begin{equation} \label{e:pre-eq_1}
\begin{split}
&\frac{d}{dt}\big(P_0(t)x,y\big)_H 
= -\big(C^*C x,y\big)_H-\big(P_0(t)x,Ay\big)_H-\big(Ax,P_0(t)y\big)_H
\\
& \myspace\qquad -\big([P_0(t)B+P_1(t,t)]\psi_1(t,t)x,y\big)_H\,.
\end{split}
\end{equation}

We seek to rewrite the term $\big([P_0(t)B+P_1(t,t)]\psi_1(t,t)x,y\big)_H$ 
so as to get rid of $\psi_1(t,t)$.
For convenience we recall here the expressions 
\begin{align*}
P_0(t)&=\int_t^T e^{(\sigma-t)A^*} C^*C \Big[\big[I - (L_t+H_t)\Lambda_t^{-1} (L_t^*+H_t^*)C^*C\big]e^{(\cdot-t)A}\Big](\sigma)\,d\sigma\,,
\\[1mm]
P_1(t,t)&=\int_t^T e^{(\sigma-t)A^*} C^*C \Big[\big[I - (L_t+H_t)\Lambda_t^{-1} (L_t^*+H_t^*)C^*C\big]\lambda(\cdot,t,t)\Big](\sigma)\,d\sigma\,,
\\[1mm]
\psi_1(t,t)&=-\big[\Lambda_t^{-1} (L_t^*+H_t^*)C^*C e^{(\cdot-t)A}\big](t)\,,
\end{align*} 
whose use leads to 
\begin{equation} \label{e:tricky-term}
\begin{split}
& -\big([P_0(t)B+P_1(t,t)]\psi_1(t,t)x,y\big)_H=
\\[1mm]
& \; =\Big(\int_t^T e^{(\sigma-t)A^*} C^*C \Big[\big[I - (L_t+H_t)\Lambda_t^{-1} (L_t^*+H_t^*)C^*C\big]\big[e^{(\cdot-t)A}B+\lambda(\cdot,t,t)\big]\Big](\sigma)\boldsymbol{\cdot}
\\
&\myspace\myspace \boldsymbol{\cdot}\big[\Lambda_t^{-1} (L_t^*+H_t^*)C^*C e^{(\cdot-t)A}x\big](t)\,d\sigma,y\Big)_H\,.
\end{split}
\end{equation}
It is useful to rewrite the term $\Lambda_t^{-1} (L_t^*+H_t^*)C^*C$ as follows:
\begin{equation} \label{e:equality}
\begin{split}
\Lambda_t^{-1} (L_t^*+H_t^*)C^*C
&=\Lambda_t^{-1} (L_t^*+H_t^*)C^*C -(L_t^*+H_t^*)C^*C+(L_t^*+H_t^*)C^*C
\\
& = (I-\Lambda_t) \Lambda_t^{-1} (L_t^*+H_t^*)C^*C+(L_t^*+H_t^*)C^*C
\\
& = - (L_t^*+H_t^*)C^*C(L_t+H_t)\Lambda_t^{-1} (L_t^*+H_t^*)C^*C+(L_t^*+H_t^*)C^*C
\\
& = (L_t^*+H_t^*)C^*C\big[I-(L_t+H_t)\Lambda_t^{-1} (L_t^*+H_t^*)C^*C\big]
\end{split}
\end{equation}
and return to \eqref{e:tricky-term} to find  
\begin{equation*}
\begin{split}
& -\big([P_0(t)B+P_1(t,t)]\psi_1(t,t)x,y\big)_H=
\\[1mm]
& \quad =\Big(\int_t^T e^{(\sigma-t)A^*} C^*C \Big[\big[I - (L_t+H_t)\Lambda_t^{-1} (L_t^*+H_t^*)C^*C\big]\big[e^{(\cdot-t)A}B+\lambda(\cdot,t,t)\big]\Big](\sigma)\boldsymbol{\cdot}
\\
&\myspace \boldsymbol{\cdot}\Big[(L_t^*+H_t^*)C^*C\big[I-(L_t+H_t)\Lambda_t^{-1} (L_t^*+H_t^*)C^*C\big] e^{(\cdot-t)A}x\Big](t)\,d\sigma,y\Big)_H
\\
&\quad=\Big(\int_t^T e^{(\sigma-t)A^*} C^*C \Big[\big[I - (L_t+H_t)\Lambda_t^{-1} (L_t^*+H_t^*)C^*C\big]\big[e^{(\cdot-t)A}B+\lambda(\cdot,t,t)\big]\Big](\sigma)\boldsymbol{\cdot}
\\
& \qquad \boldsymbol{\cdot}\int_t^T \big[B^*e^{(q-t)A^*}+\lambda(q,t,t)^*\big]C^*C
\Big[\big[I-(L_t+H_t)\Lambda_t^{-1} (L_t^*+H_t^*)C^*C\big] e^{(\cdot-t)A}x \Big](q)\,dq\,d\sigma,y\Big)_H\,.
\end{split}
\end{equation*}
Moving the integral (in $\sigma$) to the right we arrive at the more eloquent 
\begin{equation*}
\begin{split}
& -\big([P_0(t)B+P_1(t,t)]\psi_1(t,t)x,y\big)_H=
\\[1mm]
&\; =\Big(\int_t^T \big[B^*e^{(q-t)A^*}+\lambda(q,t,t)^*\big] C^*C 
\Big[\big[I - (L_t+H_t)\Lambda_t^{-1} (L_t^*+H_t^*)C^*C\big]e^{(\cdot-t)A}x\Big](q)dq,
\\
& \quad \big[B^*e^{(\sigma-t)A^*}+\lambda(\sigma,t,t)^*\big]C^*C
\Big[\big[I-(L_t+H_t)\Lambda_t^{-1} (L_t^*+H_t^*)C^*C\big] e^{(\cdot-t)A}y\Big](\sigma)\,d\sigma\Big)_H
\\
&\; =\big([B^*P_0(t)+P_1(t,t)]x,[B^*P_0(t)+P_1(t,t)]y\big)_U\,,
\end{split}
\end{equation*}
which inserted in \eqref{e:pre-eq_1} yields the first equation in \eqref{e:DRE}.

\smallskip

% eq. #2 for P_1

\noindent
{\bf 2.} We move on to the verification of the second equation in the unknown $P_1(t,p)$.
The representation \eqref{e:P_1_v2} of $P_1(t,p)$ is recorded and expanded for the reader's convenience: 
\begin{equation*}
\begin{split}
P_1(t,p)&=\int_t^T e^{(\sigma-t)A^*}C^*C Z_2(\sigma,p,t)\,d\sigma
\\
&= \int_t^T e^{(\sigma-t)A^*}C^*C\Big[\big[I-(L_t+H_t)\Lambda_t^{-1}(L_t^*+H_t^*)C^*C\big]\lambda(\cdot,p,t)\Big](\sigma)\,d\sigma\,,
\end{split}
\end{equation*}
where
\begin{equation*}
Z_2(\sigma,p,t)= \lambda(\sigma,p,t)+\big[\big(L_t+H_t\big)\psi_2(\cdot,p,t)\big](\sigma)\,,
\end{equation*}
with 
\begin{equation*}
\lambda(\sigma,p,t)=\int_t^\sigma e^{(\sigma-q)A} k(q-p)B\,dq\,,
\;\;
\psi_2(r,p,t)=- \big[\Lambda_t^{-1} \big(L_t^*+H_t^*\big)C^*C\lambda(\cdot,p,t)\big](r)\,;
\end{equation*} 
we note that in particular
\begin{equation} \label{e:D-lambda}
\partial_t \lambda(\sigma,p,t)=- e^{(\sigma-t)A} k(t-p)B\,.
\end{equation}
With $v\in U$ and $y\in \cD(A)$, we have
\begin{equation*}
\begin{split}
\partial_t \big(P_1(t,p)v,y\big)_H &=-\big(C^*C \overbrace{Z_2(t,p,t)}^{\equiv 0}v,y\big)_H
- \Big(\int_t^T e^{(\sigma-t)A^*}C^*C Z_2(\sigma,p,t)v\,d\sigma,Ay\Big)_H 
\\
&\quad + \int_t^T e^{(\sigma-t)A^*}C^*C \partial_t Z_2(\sigma,p,t)v\,d\sigma,y\Big)_H\,,
\end{split}
\end{equation*}
which taking into account Lemma~\ref{l:derivatives} yields 
\begin{equation*}
\begin{split}
& \partial_t\big( P_1(t,p)v,y\big)_H
= -(P_1(t,p)v,y)_H- \Big(\int_t^T e^{(\sigma-t)A^*}C^*C e^{(\sigma-t)A}k(t-p)Bv\,d\sigma,y\Big)_H 
\\
&\quad + \Big(\int_t^T e^{(\sigma-t)A^*}C^*C \Big[\big(L_t+H_t\big)\Lambda_t^{-1}\big(L_t^*+H_t^*\big)
C^*C\big[e^{(\cdot-t)A}B+\lambda(\cdot,t,t)\big]\psi_2(t,p,t)v\Big](\sigma)
\,d\sigma,y\Big)_H
\\
&\qquad + \Big[\big(L_t+H_t\big)\Lambda_t^{-1}\big(L_t^*+H_t^*\big)C^*Ce^{(\cdot-t)A}k(t-p)Bv\big]\Big](\sigma)\,d\sigma,y\Big)_H
\\
&\qquad - \Big(\int_t^T e^{(\sigma-t)A^*}C^*C \Big(\big[e^{(\cdot-t)A}B+\lambda(\sigma,t,t)\big]\psi_2(t,p,t)v\Big](\sigma)x\,d\sigma,y\Big)_H\,.
\end{split}
\end{equation*}
By grouping the terms which contain $e^{(\sigma-t)A}k(t-p)B$, $e^{(\cdot-t)A}B$, $\lambda(\sigma,t,t)$,
respectively, we find
\begin{equation*}
\begin{split}
& \partial_t \big(P_1(t,p)v,y\big)_H = -(P_1(t,p)v,Ay)_H
\\
% group/merge
& \qquad - \Big(\int_t^T e^{(\sigma-t)A^*}C^*C 
\Big[\big[e^{(\cdot-t)A}+ \big(L_t+H_t\big)\big]k(t-p)Bv\Big](\sigma) \,d\sigma,y\Big)_H
\\
& \qquad - \Big(\int_t^T e^{(\sigma-t)A^*}C^*C 
\Big[\big[e^{(\cdot-t)A}+ \big(L_t+H_t\big)\big] B\psi_2(t,p,t)v(\sigma)\Big]\,d\sigma,y\Big)_H
\\
& \qquad - \Big(\int_t^T e^{(\sigma-t)A^*}C^*C 
\Big[\big[\lambda(\cdot,t,t)+\big(L_t+H_t\big)\psi_2(\cdot,t,t)\big] \psi_2(t,p,t)v(\sigma)\Big]\,d\sigma,y\Big)_H
\\
% neater
& \myspace = -(P_1(t,p)x,Ay)_H 
- \Big(\int_t^T e^{(\sigma-t)A^*}C^*C Z_1(\sigma,t)k(t-p)Bv\,d\sigma,y\Big)_H
\\
& \qquad - \Big(\int_t^T e^{(\sigma-t)A^*}C^*C Z_1(\sigma,t) B\psi_2(t,p,t)v\,d\sigma,y\Big)_H
\\
& \qquad - \Big(\int_t^T e^{(\sigma-t)A^*}C^*C  Z_2(\sigma,t,t)\psi_2(t,p,t)v\,d\sigma,y\Big)_H\,,
\end{split}
\end{equation*}
that is
\begin{equation} \label{e:pre-eq_2}
\begin{split}
\partial_t \big(P_1(t,p)v,y\big)_H &= -\big(P_1(t,p)v,Ay\big)_H-\big(P_0(t)kt-p)Bv,y\big)_H
\\[1mm]
& \qquad - \big([P_0(t)B+P_1(t,t)]\psi_2(t,p,t)v,y\big)_H\,.
%- \big(P_0(t)B\psi_2(t,p,t)x,y\big) - \big(P_1(t,t)\psi_2(t,p,t)x,y\big)\,.
\end{split}
\end{equation}

The analysis on the scalar product $\big([P_0(t)B+P_1(t,t)]\psi_2(t,p,t)v,y\big)_H$ is pretty much
akin to the one carried out in the previous step:
% CHECK: per ora, calcoli mantenuti
%
with the function $\psi_1(t,t)$ now replaced by 
\begin{equation*}
\psi_2(t,p,t)=-\big[\Lambda_t^{-1} (L_t^*+H_t^*)C^*C \lambda(\cdot,p,t)\big](t)\,,
\end{equation*}
we first find
\begin{equation*}
\begin{split} 
&-\big([P_0(t)B+P_1(t,t)]\psi_2(t,p,t)v,y\big)_H =
\\
& \;\Big(\int_t^T e^{(\sigma-t)A^*}C^*C  
\Big[\big[I- \big(L_t+H_t\big)\Lambda_t^{-1}\big(L_t^*+H_t^*\big)C^*C\big]\,
\big[e^{(\cdot-t)A}B+\lambda(\cdot,t,t)\big]\Big](\sigma) \boldmath{\cdot}
\\
& \myspace \boldmath{\cdot}\big[\Lambda_t^{-1} (L_t^*+H_t^*)C^*C \lambda(\cdot,p,t)v\big](t)\,d\sigma,y\Big)_H\,.
\end{split}
\end{equation*}
Next, by recalling the equality \eqref{e:equality} we see that
\begin{equation*}
\begin{split} 
&-\big([P_0(t)B+P_1(t,t)]\psi_2(t,p,t)v,y\big)_H =
\\
& \; =\Big(\int_t^T e^{(\sigma-t)A^*}C^*C  
\Big[\big[I- \big(L_t+H_t\big)\Lambda_t^{-1}\big(L_t^*+H_t^*\big)C^*C\big]\,
\big[e^{(\cdot-t)A}B+\lambda(\cdot,p,t)\big]\Big](\sigma)\boldmath{\cdot}
\\
& \myspace \boldmath{\cdot} \Big[(L_t^*+H_t^*)C^*C\big[I-(L_t+H_t)\Lambda_t^{-1} (L_t^*+H_t^*)C^*C\big]\lambda(\cdot,p,t)v\Big](t) \,d\sigma,y\Big)_H
\\
& \; =\Big(\int_t^T e^{(\sigma-t)A^*}C^*C  
\Big[\big[I- \big(L_t+H_t\big)\Lambda_t^{-1}\big(L_t^*+H_t^*\big)C^*C\big]\,
\big[e^{(\cdot-t)A}B+\lambda(\cdot,p,t)\big]\Big](\sigma)\boldmath{\cdot}
\\
& \qquad \boldmath{\cdot} \int_t^T\big[B^*e^{(q-t)A^*}+\lambda(q,t,t)^*\big]C^*C
\big[I-(L_t+H_t)\Lambda_t^{-1} (L_t^*+H_t^*)C^*C\big]\lambda(\cdot,p,t)v\Big](q)\,dq\,d\sigma,y\Big)_H\,.
% prima di: scaricando a dx
\end{split}
\end{equation*}
A reworking of the latter scalar product leads to the conclusion
\begin{equation*}
\begin{split} 
&-\big([P_0(t)B+P_1(t,t)]\psi_2(t,p,t)v,y\big)_H =
\\
& \; = \Big(\int_t^T\big[B^*e^{(q-t)A^*}+\lambda(q,t,t)^*\big]C^*C
\big[I-(L_t+H_t)\Lambda_t^{-1} (L_t^*+H_t^*)C^*C\big]\lambda(\cdot,p,t)v\Big](q)\,dq,
\\
& \qquad \int_t^T\Big[\big[B^*e^{(\cdot-t)A^*}+\lambda(\cdot,p,t)^*\big]C^*C
\big[I-(L_t+H_t)\Lambda_t^{-1} (L_t^*+H_t^*)C^*C\big]e^{(\cdot-t)A}y\Big](\sigma)\,d\sigma\Big)_U
\\
& \; =\big([B^*P_1(t,p)+P_2(t,p,t)]v,[B^*P_0(t)+P_1(t,t)]y\big)_U\,,
\end{split}
\end{equation*} 
which combined with \eqref{e:pre-eq_2} validates the second equation of the system \eqref{e:DRE}.
 
\smallskip

% eq. #3 for P_2

\noindent 
{\bf 3.}
In order to prove that the third equation in the unknown $P_2(t,p,q)$ (within the coupled system
\eqref{e:DRE}) is valid, as first thing we need to recall the representation \eqref{e:P_2_v2} of $P_2(t,p,q)$, as well as the derivatives $\partial_t \lambda(\sigma,q,t)^*$, $\partial_t Z_2(\sigma,p,t)$ and $\partial_t \psi_2(r,p,t)$ as from \eqref{e:D-lambda}, \eqref{e:D-Z_2}, \eqref{e:D-psi_2}, respectively.
We have for $u,v\in U$
\begin{equation*}
\begin{split}
& \partial_t \big(P_2(t,p,q)u,v\big)_U
=-\big(\lambda(t,q,p)^*C^*C \overbrace{Z_2(t,p,t)}^{\equiv 0}u,v\big)_U
\\
& \myspace 
- \Big(\int_t^T e^{(\sigma-t)A^*}C^*C K(t-q)^*C^*C Z_2(\sigma,p,t)u\,d\sigma ,v\Big)_U  
\\
&\qquad 
- \Big(\int_t^T \lambda(\sigma,q,t)^*C^*C e^{(\sigma-t)}k(t-p)Bu\,d\sigma,v\Big)_U\,,
\\
&\qquad 
+ \Big(\int_t^T \lambda(\sigma,q,t)^* C^*C \Big[\big(L_t+H_t\big)\Lambda_t^{-1}\big(L_t^*+H_t^*\big)C^*C \boldmath{\cdot}
\\
& \myspace
\boldmath{\cdot}\Big[\big[e^{(\cdot-t)A}B+\lambda(\cdot,t,t)\big]\psi_2(t,p,t)
+e^{(\cdot-t)A}k(t-p)Bu\big]\Big](\sigma)\,d\sigma,v\Big)_U
\\
&\qquad - \Big(\int_t^T \lambda(\sigma,q,t)^*  C^*C \big[e^{(\sigma-t)A}B+\lambda(\sigma,t,t)\big]\psi_2(t,p,t)u\, d\sigma ,v\Big)_U\,,
\end{split}
\end{equation*}
which -- after a rearrangement of terms -- becomes
\begin{equation*}
\begin{split}
& \partial_t \big(P_2(t,p,q)u,v\big)_U = -(B^*k(t-p)^*P_1(t,p)u,v)_U
\\
% group/merge
& \; -\Big(\int_t^T \lambda(\sigma,q,t)^* C^*C\Big[\big[I-\big(L_t+H_t\big)\Lambda_t^{-1}\big(L_t^*+H_t^*\big)C^*C\big] e^{(\cdot-t)A}\Big](\sigma)\,d\sigma \,k(t-p)Bu,v\Big)_U
\\
& \; - \Big(\int_t^T \lambda(\sigma,q,t)^* C^*C\Big[\big[I-\big(L_t+H_t\big)\Lambda_t^{-1}\big(L_t^*+H_t^*\big)C^*C\big] e^{(\cdot-t)A}B\Big](\sigma)\,d\sigma \,\psi_2(t,p,t)x,y\Big)
\\
& \; - \Big(\int_t^T \lambda(\sigma,q,t)^* C^*C\Big[\big[I-\big(L_t+H_t\big)\Lambda_t^{-1}\big(L_t^*+H_t^*\big)C^*C\big] \lambda(\cdot,t,t)\Big](\sigma)\,d\sigma\, \psi_2(t,p,t)x,y\Big)
\\
& \myspace\quad=: \sum_{i=0}^3 T_i\,,
\end{split}
\end{equation*}
where the summands $T_i$, $i\in \{1,2,3\}$, can be rewritten -- taking into account the expression \eqref{e:P_1*} for $P_1(t,q)^*$ -- as follows, respectively: 
\begin{equation*}
\begin{cases}
T_1= -\big(P_1(t,q)^*k(t-p)Bu,v\big)_U
\\[1mm]
T_2= -\big(P_1(t,q)^*B\psi_2(t,p,t)u,v\big)_U
\\[2mm]
T_3= -\Big(\displaystyle \int_t^T \lambda(\sigma,q,t)^* C^*C Z_2(\sigma,t,t)\,d\sigma \psi_2(t,p,t)u,v\Big)_U
= -\big(P_2(t,t,q)\psi_2(t,p,t)u,v\big)_U\,.
\end{cases}
\end{equation*}
So far, we attained 
\begin{equation} \label{e:pre-eq_3}
\begin{split}
\partial_t \big(P_2(t,p,q)u,v\big)_U &= -\big(B^*k(t-p)^* P_1(t,p)u,v\big)_U
-\big(P_1(t,q)^*k(t-p)Bu,v\big)_U
\\[1mm]
& \qquad - \big([P_1(t,q)^*B+P_2(t,t,q)]\psi_2(t,p,t)u,v\big)_U\,.
\end{split}
\end{equation}
Once again, in order to disclose that the equation \eqref{e:pre-eq_3} is actually the soughtafter one
comprised in \eqref{e:DRE}, it remains to pinpoint the term $\big([P_1(t,q)^*B+P_2(t,t,q)]\psi_2(t,p,t)u,v\big)_U$.
We have
\begin{equation*}
\begin{split} 
&-\big([P_1(t,q)^*B+P_2(t,t,q)]\psi_2(t,p,t)u,v\big)_U=
\\
& \; =\Big(\int_t^T \lambda(\sigma,q,t)^* C^*C  
\Big[\big[I- \big(L_t+H_t\big)\Lambda_t^{-1}\big(L_t^*+H_t^*\big)C^*C\big]\,
\big[e^{(\cdot-t)A}B+\lambda(\cdot,t,t)\big]\,\boldmath{\cdot}
\\
& \myspace
\boldmath{\cdot} \big[\Lambda_t^{-1}\big(L_t^*+H_t^*\big)C^*C\lambda(\cdot,p,t)u\big](t)\Big](\sigma)\,d\sigma ,v\Big)_U
% riscrivo
\\
& \; =\Big(\int_t^T \lambda(\sigma,q,t)^* C^*C  
\Big[\big[I- \big(L_t+H_t\big)\Lambda_t^{-1}\big(L_t^*+H_t^*\big)C^*C\big]
\big[e^{(\cdot-t)A}B+\lambda(\cdot,t,t)\big]\Big]\boldmath{\cdot}
\\
& \myspace \boldmath{\cdot} \big(L_t^*+H_t^*\big)C^*C\Big[\big[I- \big(L_t+H_t\big)\Lambda_t^{-1}\big(L_t^*+H_t^*\big)C^*C\big]\,\lambda(\cdot,p,t)u\Big](t)\,d\sigma,v\Big)_U
% riscrivo
\\
& \; =\Big(\int_t^T \lambda(\sigma,q,t)^* C^*C  
\Big[\big[I- \big(L_t+H_t\big)\Lambda_t^{-1}\big(L_t^*+H_t^*\big)C^*C\big]
\big[e^{(\cdot-t)A}B+\lambda(\cdot,t,t)\big]\Big]\boldmath{\cdot}
\\
& \qquad\qquad \boldmath{\cdot} \big[B^*e^{(r-t)A^*}+\lambda(r,t,t)^*\big]
C^*C\Big[\big[I-\big(L_t+H_t\big)\Lambda_t^{-1}\big(L_t^*+H_t^*\big)C^*C\big]\,
\lambda(\cdot,p,t)u\Big](r)\,dr\,d\sigma,v\Big)_V\,,
\end{split}
\end{equation*}
which gives 
\begin{equation*}
\begin{split} 
&- \big([P_1(t,q)^*B+P_2(t,t,q)]\psi_2(t,p,t)u,v\big)_U=
\\
& \quad =\Big(\int_t^T \big[B^*e^{(r-t)A^*}+\lambda(r,t,t)^*\big]
C^*C\Big[\big[I- \big(L_t+H_t\big)\Lambda_t^{-1}\big(L_t^*+H_t^*\big)C^*C\big]\,\lambda(\cdot,p,t)u\Big](r)\,dr,
\\
& \qquad\qquad 
\int_t^T \big[B^*e^{(\sigma-t)A^*}+\lambda(\sigma,t,t)^*\big]C^*C\Big[\big[I-\big(L_t+H_t\big)\Lambda_t^{-1}\big(L_t^*+H_t^*\big)C^*C\big]\,\lambda(\cdot,q,t)y\Big](\sigma)\,d\sigma \Big)_U
\\
& \quad = \big([B^*P_1(t,p)+P_2(t,p,t)]u,[B^*P_1(t,q)+P_2(t,q,t)]v\big)_U
\end{split}
\end{equation*}
for all $x,y\in Y$.
The third equation of the system \eqref{e:DRE} is thus established combining the latter result
with \eqref{e:pre-eq_3}. 

% END of PROOF
\qed

% SECTION 4: UNIQUENESS.

\section{Proof of statement S6: uniqueness for the Riccati-type equation}
\label{s:uniqueness}
In the previous section we showed that the triplet $(P_0(t),P_1(t,p),P_2(t,p,q)$) -- with $P_i$ defined
by \eqref{e:riccati-ops_2}, $i\in \{0,1,2\}$ -- is a solution to the coupled system \eqref{e:DRE}. 
However, in order to initiate the synthesis of the optimal control, we need to demonstrate that the 
coupled system \eqref{e:DRE} is actually {\em uniquely} solvable.
This may be accomplished more readily by relating system \eqref{e:DRE} to a quadratic equation
in the space $Y_t=H\times L^2(0,t;U)$.
% CHECK: $H\times L^2(0,t;H)$ is wrong}

\begin{theorem} 
With reference to the optimal control problem \eqref{e:mild-sln_s}-\eqref{e:cost_s} under the Assumptions~\ref{a:ipo_0} and the hypothesis \eqref{e:ipo_1}, let $P_0(t)$, $P_1(t,\cdot)$ and $P_2(t,\cdot,:)$ as from the statement S3. of Theorem~\ref{t:main}.
% \eqref{e:riccati-ops_2}
Then, the (matrix) operator $P(t)$ defined by\footnote{The symbols $\cdot$ and $:$ stand for hidden variables.} 
\begin{equation} \label{e:big-op}
P(t):=
\begin{pmatrix}
P_0(t) & P_1(t,\cdot)
\\
P_1(t,:)^* & P_2(t,\cdot,:)
\end{pmatrix}
\end{equation}
is the unique solution of the following quadratic equation:
\begin{equation} \label{e:big-DRE}
\hspace{7mm}
\frac{d}{dt} P(t)=-Q -P(t)[\cA+\cK_1(t)] -[\cA^*+\cK_2(t)]P(t)
+ P(t)\cI_{1,t}\cB\cI_{2,t}P(t)\,,
\end{equation}
supplemented with the final condition $P(T)=0$.
% BIG OPERATORS 
The coefficients read as  
\begin{equation*}
\begin{split}
&\cQ:=
\begin{pmatrix} 
C^*C & 0
\\
0 & 0
\end{pmatrix},
\qquad
\cA:=
\begin{pmatrix} 
A & 0
\\
0 & 0
\end{pmatrix},
\qquad
\cB:=
\begin{pmatrix} 
BB^* & B
\\
B & I
\end{pmatrix}
\\[1mm]
& \cK_1(t):=
\begin{pmatrix} 
0 & k(t-\cdot)
\\
0 & 0
\end{pmatrix},
\qquad
\cK_2(t):=
\begin{pmatrix} 
0 & 0
\\
k(t-\!:)^* & 0
\end{pmatrix}, 
\\[1mm]
& \cI_{1,t}:=
\begin{pmatrix} 
I & 0
\\
0 & \chi_{\{t\}}(\cdot)
\end{pmatrix}, 
\qquad
\cI_{2,t}:=
\begin{pmatrix} 
I & 0
\\
0 & \chi_{\{t\}}(:)
\end{pmatrix}
\end{split}
\end{equation*}
($\chi_{\{t\}}$ denotes the characteristic function of a set, here specifically the singleton 
$\{t\}$).
\end{theorem}

\begin{proof}
We retrace the major steps of the proof of \cite[Theorem~1.6]{ac-bu-memory1_2024}.
{\bf 1.} To confirm the natural claim that the matrix operator $P(t)$ defined by \eqref{e:big-op} is a
solution to the equation \eqref{e:big-DRE}, a subtle technical issue must be addressed beforehand,
see \cite[Proposition~6.1]{ac-bu-memory1_2024}. 
(Indeed, one thing is to compute e.g. the derivative $\frac{\partial}{\partial t} \big(P_1(t,p)v,y\big)_H$ for $v\in U$, $y\in \cD(A)$; 
computing $\frac{\partial}{\partial t} \big(P_1(t,\cdot)f(\cdot),y\big)$, with $f\in L^2(0,t;U)$, $y\in \cD(A)$, another.)

\smallskip
\noindent
{\bf 2.} 
The proof of uniqueness is similar as well to the earlier one in \cite[Proposition~6.1]{ac-bu-memory1_2024};
we detail the computations which are specific to the problem at hand, after recalling some introductory elements.

% We introduce the notation $\cH_t$ for the family of spaces $H\times L^2(0,t;H)$.
%
The operator $P(\tau)$ defined by \eqref{e:big-op} belongs to the space $\cL(\cH_\tau)$ for
any $\tau\in (0,T]$, and hence $P(\cdot)$ belongs to the Banach space
\begin{equation*}
Z:=\big\{Q(\cdot)\colon Q(\tau)\in \cL(Y_\tau), \; \tau\in (0,T]\big\}, 
\quad \|Q\|_Z= \sup_{\tau\in (0,T]} \|Q(\tau)\|_{\cL(Y_\tau)}<\infty\,.
\end{equation*}
By construction, $Y_t\subseteq Y_\tau$ for $\tau<t$.
 
We proceed by contradiction: let $Q(\cdot)\in Z$ be another solution to \eqref{e:big-DRE}
(besides $P(\tau)$ defined by \eqref{e:big-op}), and set $V:=P-Q$. 
Of course, we have as well
\begin{equation*}
Q(\tau)=
\begin{pmatrix}
Q_0(\tau) & Q_1(\tau,\cdot)
\\
Q_1(\tau,:)^* & Q_2(\tau,\cdot,:)
\end{pmatrix}, 
\qquad
V(\tau)=
\begin{pmatrix}
V_0(\tau) & V_1(\tau,\cdot)
\\
V_1(\tau,:)^* & V_2(\tau,\cdot,:)
\end{pmatrix}, 
\quad \tau\in (0,T).
\end{equation*}
For given $\tau, t\in (0,T)$, with $\tau<t$, the difference operator $V(\cdot)$ satisfies in $[t,T)$ the differential equation
\begin{equation}\label{e:DRE_uni} 
\begin{split}
& \frac{d}{dr} V(r)=[\cA^*+\cK_2(r)]V(r) -V(r)[\cA+\cK_1(r)] 
\\
& \myspace \myspace
+ V(r)\cI_{1,r}\cB \cI_{2,r}P(r) - Q(r)\cI_{1,r}\cB \cI_{2,r}V(r)=0\,,
\end{split}
\end{equation}
with $V(T)=P(T)-U(T)=0$.
We remark that the equation \eqref{e:DRE_uni} is to be interpreted as follows: the first term in its left hand side is 
\begin{equation*} 
\frac{d}{dr} \big(V(r)X_0,X_1\big)_{Y_t}
\end{equation*} 
(with $X_0, X_1\in Y_t$); the other terms are understood in a similar way.

We note that the operator $\cA$ is the infinitesimal generator of a $C_0$-semigroup 
$\{e^{s\cA}\}_{s\ge 0}$ in $Y_\tau$ (for any $\tau\in (0,T)$), explicity given by
\begin{equation*}
e^{s\cA} = 
\begin{pmatrix}
e^{sA} & 0
\\
0 & I
\end{pmatrix};
\end{equation*}
the bound $\|e^{s\cA}\|_{\cL(Y_\tau)} \le C_T$ holds true for any $s\in [0,T]$ and $\tau\in (0,T)$.

Thus, the mild form of \eqref{e:DRE_uni} in $[t,T]$ is
\begin{equation} \label{e:mildV} 
\begin{split}
& V(r) = \int_r^T e^{(p-r)\cA^*}\Big[V(p)\cK_1(p)+\cK_2(p)V(p)
\\
& \myspace \myspace
-V(p)I_{1,p}\cB I_{2,p}P(p) - Q(p)I_{1,p}\cB I_{2,p}V(p)\Big]e^{(p-r)\cA}\,dp
\end{split}
\end{equation}
(this mild form is understood as \eqref{e:DRE_uni} did).
We can now observe that it makes sense to estimate $\|V(r)\|_{\cL(Y_t)}$ when $r\ge t$.    
In addition, %It
it can be shown that for $r\ge t$ the map $r \longmapsto \|V(r)\|_{\cL(Y_t)}$ is lower semi-continuous and hence it is a measurable function in $[t,T]$ (see Lemma~\ref{l:lsc} at the end of this section, whose proof is omitted).

Starting from \eqref{e:mildV}, we find 
\begin{equation*}
\begin{split} 
\|V(r)\|_{\cL(Y_t)} 
&\le \int_r^T  C_T^2 \Big[2 \|V(p)\|_{\cL(Y_t)}\|k\|_{L^2(0,T)} 
+ \|V(p)\|_{\cL(Y_t)} \|\cB\|\,\|P(p)\|_{\cL(Y_t)} 
\\
& \qquad  \|Q(p)\|_{\cL(Y_t)} \|\cB\|\,\|V(p)\|_{\cL(Y_t)}\Big]\,dp
\\
& \le C_1 \int_r^T  \|V(p)\|_{\cL(Y_t)}\,dp\,, 
\end{split}
\end{equation*}
where $C_1$ is a suitable positive constant depending on $T$, $1+\|B\|^2$, $\|k\|_{L^2(0,T)}$, 
$\|P\|_Z$, $\|Q\|_Z$. 

By the Gronwall Lemma it follows that $\|V(r)\|_{\cL(Y_t)} =0$ for $r\in[t,T]$; 
in particular, $\|V(t)\|_{\cL(Y_t)}=0$. 
Since $t>\tau$ was given and yet arbitrary, it follows that 
\begin{equation*}
V_0(t)=0, \quad V_1(t,s)=0, \quad V_2(t,s,q)=0 \quad \forall s,q\in [0,\tau], 
\; \forall t\in [\tau,T]\,.
\end{equation*}
This means that the operators $Q$ and $P$ coincide for any $t,s,q$ such that $s,q \in [0,\tau]$ and $t\in [\tau,T]$ (with arbitrary $\tau\in (0,T$)).

Thus, this second part of the proof is complete once the following result is established; {\em cf.}
\cite[Lemma~6.2]{ac-bu-memory1_2024}:

% LEMMA (key for uniqueness)

\begin{lemma} \label{l:lsc} 
For any $\tau\in (0,T)$, the map $t \longmapsto \|V(t)\|_{\cL(Y_\tau)}$ is lower semi-continuous in
$[\tau,T]$.
\end{lemma}

\end{proof}

\medskip

% ACKNOWLEDGEMENTS

{\small
\section*{Acknowledgements}
The research of the second author has been performed in the framework of the MIUR-PRIN Grant 2020F3NCPX ``Mathematics for Industry 4.0 (Math4I4)''.
She is currently supported by the Universit\`a degli Studi di Firenze under the 2024 Project {\em ``Controllo lineare-quadratico per equazioni integro-differenziali a 
derivate parziali''}, of which she is responsible.
She is also a member of the Gruppo Nazionale per l'Analisi Mate\-ma\-tica, la Probabilit\`a  e le loro Applicazioni of the Istituto Nazionale di Alta Matematica (GNAMPA-INdAM) and participant to
the 2024 GNAMPA Project ``Controllo ottimo infinito dimensionale: aspetti teorici ed applicazioni''.
}

% APPENDICI

\appendix

% Appendix A

\section*{Appendix. Proof of statement S2: transition properties} %of the optimal pair} 
%\label{a:app_a}
In this Appendix we prove the statement S2. of Thorem~\ref{t:main}, namely, that both the optimal control and the component $w(\cdot)$ of the optimal state -- that comprises the value of the evolution at time $s$ and the memory of the control up to time $s$ -- satisfy the soughtafter (respective) transition properties. 

We prove a transition property for the optimal control as first.
\\
{\bf 1.}
One needs to recall that
\begin{equation*} 
\hat{u}(\cdot,s,X_0)= \underset{u\in L^2(s,T;U)}{\arg\min} J_s(u)
\end{equation*}
% CHECK
where $X_0=(w_0,\eta)^T$, with $\eta(\cdot):=\hat{u}(\cdot,0,w_0)\big|_{[0,s]}$.
The starting point is the formula \eqref{e:start-for-transition} brought about by the optimality condition, that reads as 
\begin{equation*}
\begin{split} 
&\hat{u}(t,s,X_0)+\big[\big(L_s^*+H_s^*\big)C^*C\big(L_s+H_s\big)\hat{u}(\cdot,s,X_0)\big](t)
\\
& \myspace = -\big[\big(L_s^*+H_s^*\big)C^*C\big[e^{A(\cdot-s)}w_0+\cK_s\eta(\cdot)\big](t)\,,
\qquad \text{a.e. in $[s,T]$.}
\end{split}
\end{equation*}
% where we need to recall that $X_0=(w_0,\eta)^T$, with $\eta(\cdot):=\hat{u}(\cdot,0,w_0)$.
%
At a subsequent time $\tau$, namely, if $0<s<\tau<T$, it holds 
\begin{equation*}
\begin{split} 
& \hat{u}(t,\tau,X_1)+\big[\big(L_\tau^*+H_\tau^*\big)C^*C\big(L_\tau+H_\tau\big)\hat{u}(\cdot,\tau,X_1)\big](t)
\\
& \myspace = -\big[\big(L_\tau^*+H_\tau^*\big)C^*C\big[e^{A(\cdot-\tau)}w_1
+\cK_s\theta(\cdot)\big](t)\,,
\end{split}
\end{equation*}
where we set 
\begin{equation*}
w_1=\hat{w}(\tau,s,w_0)\,, \quad  \theta(\cdot):=\hat{u}(\cdot,0,w_0)\big|_{[0,\tau]}\,.
\end{equation*}

Subtracting the latter equality from the preceding one yields, for any $t\in [\tau,T]$, 
\begin{equation*} %\label{e:RHS}
\begin{split} 
&\hat{u}(t,s,X_0)-\hat{u}(t,\tau,X_1)
\\
& \, =-\big[\big(L_s^*+H_s^*\big)C^*C\big(L_s+H_s\big)\hat{u}(\cdot,s,X_0)
-\big(L_\tau^*+H_\tau^*\big)C^*C\big(L_\tau+H_\tau\big)\hat{u}(\cdot,\tau,X_1)\big](t)
\\
& \quad -\Big[\big(L_s^*+H_s^*\big)C^*C\big[e^{A(\cdot-s)}w_0+\cK_s\eta(\cdot)\big]
- \big(L_\tau^*+H_\tau^*\big)C^*C\big[e^{A(\cdot-\tau)}w_1+\cK_\tau \theta(\cdot)\big]\Big](t)\,.
\\
& 
\, =-\Big[\big(L_\tau^*+H_\tau^*\big)C^*C\big[\big(L_s+H_s\big)\hat{u}(\cdot,s,X_0)
-\big(L_\tau+H_\tau\big)\hat{u}(\cdot,\tau,X_1)\big]\Big](t)
\\
& 
\quad -\Big[\big(L_\tau^*+H_\tau^*\big) C^*C \big[e^{A(\cdot-s)}w_0-e^{A(\cdot-\tau)}w_1
+\cK_s\eta(\cdot)-\cK_\tau \theta(\cdot)\big]\Big](t)\,.
\end{split}
\end{equation*}
Here we note two properties that are crucially utilized in the computations: the first is the
following expression of $H_su(t)$ -- different from \eqref{e:operator_Hs} -- obtained via the
Fubini's theorem:
\begin{equation*}
H_s u(t) = \int_s^t \int_p^t e^{A(t-q)}k(q-p)\,dq \,Bu(p)\,dp\,;
\end{equation*}
the second is the fact that the adjoint operators
% $L_s^*$ and $H_s^*$ defined as 
%\begin{equation*}
%[L_s^*v](t)= \int_t^T B^* e^{A^*(q-t)}v(q)\, dq\,, \;
%[H_s^*v](t)= \int_t^T \int_t^q B^* e^{A^*(q-p)} k(p-q)^*\,dp v(q)\, dq\,, \quad s<t<T
%\end{equation*}
\begin{equation*}
\begin{split} 
& [L_s^*v](t)= \int_t^T B^* e^{A^*(q-t)}v(q)\, dq\,, \qquad s<t<T\,,
\\
& [H_s^*v](t)= \int_t^T \int_t^q B^* e^{A^*(q-p)} k(p-q)^*\,dp v(q)\, dq\,, \qquad s<t<T
\end{split}
\end{equation*}
do not depend on $s$ (they are just defined for any $t$ which is not smaller than $s$).
We push forward the computations in order to obtain a simplified form of the right hand side
of the above identity.
An outline of the major points is provided: 
% (i)
we decompose the term $(L_s+H_s\big)\hat{u}(\cdot,s,X_0)]$ -- that corresponds to an integral between $s$ and $\sigma$, say -- 
as the sum of two integrals, over $(s,\tau)$ and $(\tau,\sigma)$ respectively;
% (ii)
the difference $e^{A(\cdot-s)}w_0-e^{A(\cdot-\tau)}w_1$ is rewritten using basic semigroups properties, whereas 
% (iii)
in order to obtain a neater expression of the difference
$[\cK_s\eta](\cdot)-[\cK_\tau \theta](\cdot)$ we set 
\begin{equation*}
\cG(\cdot,q,p):=e^{A(\cdot-q)} k(q-p)\,B\,.
\end{equation*}
The latter yields
\begin{equation*}
\begin{split}
&[\cK_s\eta](\cdot)-[\cK_\tau \theta](\cdot)
= \int_0^s \int_s^\cdot \cG(\cdot,q,p)\eta(p)\,dq\,dp
- \int_0^\tau \int_\tau^\cdot \cG(\cdot,q,p)\theta(p)\,dq\,dp
\\
&\quad = \Big[\int_0^s\int_s^\tau 
+ \underline{\int_0^s\int_\tau^\cdot}\Big] \cG(\cdot,q,p)\eta(p)\,dq\,dp
- \Big[\underline{\int_0^s\int_\tau^\cdot} 
+ \int_s^\tau\int_\tau^\cdot\Big] \cG(\cdot,q,p)\theta(p)\,dq\,dp\,,
\end{split}
\end{equation*}
where the underlined integrals clearly cancel, since $\eta\equiv \theta$ on $[0,s]$.
Thus, the difference simplifies as follows:
\begin{equation*}
[\cK_s\eta](\cdot)-[\cK_\tau \theta](\cdot)
= \int_0^s\int_s^\tau \cG(\cdot,q,p)\eta(p)\,dq\,dp
- \int_s^\tau\int_\tau^\cdot \cG(\cdot,q,p)\theta(p)\,dq\,dp\,.
\end{equation*} 
Taking into account all of the above, we achieve
\begin{equation} \label{e:pre-final}\tag{*}
\begin{split} 
&\hat{u}(t,s,X_0)-\hat{u}(t,\tau,X_1) = -\big(L_\tau^*+H_\tau^*\big)C^*C \big(L_\tau+H_\tau\big)\big[\hat{u}(\cdot,s,X_0)-\hat{u}(\cdot,\tau,X_1)\big]
\\[1mm]
& \quad
-\big(L_\tau^*+H_\tau^*\big)C^*C \Big[
\int_s^\tau e^{A(\cdot-p)}B\theta(p)\,dp 
+ \int_s^\tau\int_p^\cdot e^{A(\cdot-q)} k(q-p)\,dq\,B\theta(p)\,dp
\\[1mm]
& \qquad\quad
+ e^{A(\cdot-\tau)}\big(e^{A(\tau-s)}w_0-w_1\big)
\\[1mm]
& \qquad\qquad
+ \int_0^s\int_s^\tau \cG(\cdot,q,p)\eta(p)\,dq\,dp
- \int_s^\tau\int_\tau^\cdot \cG(\cdot,q,p)\theta(p)\,dq\,dp\Big](t)
%+\int_0^s\int_s^\tau e^{A(\cdot-q)} k(q-p)\,dq\,B\eta(p)\,dp
%- \int_s^\tau\int_\tau^\cdot e^{A(\cdot-q)} k(q-p)\,dq\,B\theta(p)\,dp\big]\Big](t)
\end{split}
\end{equation} 

Moving the first summand in the right hand side of \eqref{e:pre-final} to the left hand side and recalling the explicit expression of $\cG(\cdot,q,p)$, we end up with 
\begin{equation*}
\begin{split} 
&\Big[I+\big(L_\tau^*+H_\tau^*\Big)C^*C \big(L_\tau+H_\tau\big)\Big]
\big(\hat{u}(t,s,X_0)-\hat{u}(t,\tau,X_1)\Big)(t)
\\
&\; = -\Big[\big(L_\tau^*+H_\tau^*\big)C^*C \Big[
\int_s^\tau e^{A(\cdot -p)}B\theta(p)\,dp + \int_s^\tau \int_p^\tau e^{A(\cdot -q)}k(q-p)\,dq B\theta(p)\,dp 
\\
&\;\quad + \cancel{\int_s^\tau\int_\tau^\cdot e^{A(\cdot -q)} k(q-p)\,dq \,B\theta(p)\,dp} + e^{A(\cdot-s)}w_0-w_1
\\
&\;\quad +\int_0^s\int_s^\tau e^{A(\cdot-q)}k(q-p)\,dq \,B\theta(p)\,dp - \cancel{\int_s^\tau \int_\tau^\cdot e^{A(\cdot-q)}k(q-p)\,dq \,B\theta(p)\,dp}\Big]\Big](t)
\\
&\; =-\Big[\big(L_\tau^*+H_\tau^*\big)C^*C e^{A(\cdot-\tau)}\big[ L_s\theta(\tau)+ H_s\theta(\tau)+e^{A(\tau-s)}w_0-w_1 +\cK_s\eta(\tau)\big]\Big](t) 
\\
&\; = -\Big[\big(L_\tau^*+H_\tau^*\big)C^*C e^{A(\cdot-\tau)}\big(\hat{w}(\tau,s,X_0)-w_1\big)\Big](t)
\\
&\;  = -\Big[\big(L_\tau^*+H_\tau^*\big)C^*C \big[e^{A(\cdot-\tau)}0\big]\Big](t)\equiv 0\,.
\end{split}
\end{equation*}
Since the operator $I+\big(L_\tau^*+H_\tau^*\big)C^*C \big(L_\tau+H_\tau\big)$ is invertible,
it follows that $\hat{u}(t,\tau,X_1)=\hat{u}(t,s,X_0)$, which concludes the first part of the 
proof.

\smallskip
\noindent
{\bf 2.}
We next show that the first component of the optimal state inherits from the optimal control 
a transition property, as well.
We have 
\begin{equation*}
\begin{split} 
& \hat{w}(t,\tau,X_0)= e^{A(t-\tau)}\hat{w}(\tau,s,X_0)
+L_\tau u\hat{(}\cdot,\tau,\hat{w}(\tau,s,X_0))(t)+ H_\tau \hat{u}(\cdot,\tau,w(\tau,s,X_0))(t)
\\
& \qquad
+ \cK_\tau\theta(t)
\\
&\quad = e^{A(t-\tau)}\hat{w}(\tau,s,X_0)+L_\tau \hat{u}(\cdot,s,w_0)(t) 
+ H_\tau \hat{u}(\cdot,s,w_0)(t)
\\
& \qquad
+ \int_\tau^t \int_0^\tau e^{A(t-q)}k(q-p)B\theta(p)\,dp\,dq
\\
&\quad = e^{A(t-\tau)}\Big[e^{A(\tau-s)}w_0+L_s \hat{u}(\cdot,s,w_0)(t) + H_s \hat{u}(\cdot,s,w_0)(t)
\\
& \qquad 
+ \int_s^\tau \int_0^s e^{A(\tau-q)}k(q-p)B\eta(p)\,dp\,dq\Big] 
\\
&\quad
+ L_\tau \hat{u}(\cdot,s,w_0)(t) + H_\tau \hat{u}(\cdot,s,w_0)(t) 
+ \int_\tau^t \int_0^\tau e^{A(t-q)}k(q-p)B\theta(p)\,dp\,dq
\\
&\quad = e^{A(t-s)}w_0 + \big[e^{A(t-\tau)} L_s \hat{u}(\cdot,s,w_0)(t)
+ L_\tau \hat{u}(\cdot,s,w_0)(t) \big] 
\\
&\qquad
+ \big[e^{A(t-\tau)} H_s \hat{u}(\cdot,s,w_0)(t)+ H_\tau \hat{u}(\cdot,s,w_0)(t)\big] 
\\
&\quad
+ \Big[e^{A(t-\tau)} \int_s^\tau \int_0^s e^{A(\tau-q)}k(q-p)B\eta(p)\,dp\,dq
+ \int_\tau^t \int_0^\tau e^{A(t-q)}k(q-p)B\theta(p)\,dp\,dq\Big]
\\
&\; =: \sum_{i=1}^4 S_i\,.
\end{split}
\end{equation*}
 
The first summand reads as $S_1= e^{A(t-s)}w_0$, while
\begin{equation*}
S_2=\int_s^\tau e^{A(t-p)}B\hat{u}(p,s,w_0)\,dp+ \int_\tau^t e^{A(t-p)}B\hat{u}(p,s,w_0)\,dp
=L_s \hat{u}(\cdot,s,w_0)(t)\,.
\end{equation*}
As for the third and fourth summands, we find
\begin{equation*}
\begin{split} 
S_3 &= \int_s^\tau \int_p^\tau e^{A(t-q)}k(q-p)B\hat{u}(p,s,w_0)\,dp
+ \int_\tau^t \int_p^t e^{A(t-q)}k(q-p)B\hat{u}(p,s,w_0)\,dp
\\
&= \int_s^\tau \int_p^t e^{A(t-q)}k(q-p)B\hat{u}(p,s,w_0)\,dp 
- \int_s^\tau \int_\tau^t e^{A(t-q)}k(q-p)B\hat{u}(p,s,w_0)\,dp
\\
&\quad
+ \int_\tau^t \int_p^t e^{A(t-q)}k(q-p)\, dq B\hat{u}(p,s,w_0)\,dp
\\
&=H_s u(\cdot,s,w_0)(t) - \int_s^\tau \int_\tau^t e^{A(t-q)}k(q-p)B\hat{u}(p,s,w_0)\,dp\,,
\end{split}
\end{equation*} 
and
\begin{equation*}
\begin{split} 
S_4 &= \int_s^\tau \int_0^s e^{A(t-q)}k(q-p)B\eta(p)\,dp\,dq
+ \int_\tau^t \int_0^\tau e^{A(t-q)}k(q-p)B\theta(p)\,dp\,dq
\\
&= \int_s^\tau \int_0^s e^{A(t-q)}k(q-p)B\eta(p)\,dp\,dq
+ \int_\tau^t \int_0^s e^{A(t-q)}k(q-p)B\eta(p)\,dp\,dq
\\
&\qquad 
+ \int_\tau^t \int_s^\tau e^{A(t-q)}k(q-p)B\theta(p)\,dp\,dq
\\
&= \int_s^\tau \int_0^s e^{A(t-q)}k(q-p)B\eta(p)\,dp\,dq
+ \int_\tau^t \int_s^\tau e^{A(t-q)}k(q-p)B\theta(p)\,dp\,dq
\\
&=\cK_s\eta(t) +\int_s^\tau \int_\tau^t e^{A(t-q)}k(q-p)\,dq B\theta(p)\,dp\,.
\end{split}
\end{equation*}  
Combining the above expression of $S_i$, $i\in \{1,2,3,4\}$, we finally attain
\begin{equation*}
\begin{split} 
&\hat{w}(t,\tau,w(\tau,s,w_0))=S_1+S_2+S_3+S_4
\\
&\quad = e^{A(t-s)}w_0+L_s\hat{u}(\cdot,s,w_0)(t) + H_s \hat{u}(\cdot,s,w_0)(t) +\cK_s\eta(t)
\\
&\qquad 
+\int_s^\tau \int_\tau^t e^{A(t-q)}k(q-p)\,dq Bu(p,s,w_1)\,dp= \hat{w}(t,s,w_0)
\end{split}
\end{equation*}  
as desired.

% Appendix B

% \section{Instrumental results} %Proof of Lemma~\ref{l:derivatives} 
%\label{a:app_b}

%\vspace{2cm}

% REFERENCES

% THE END
\end{document}